\documentclass[twoside,leqno,symbols-for-thanks]{rmi}

\usepackage[colorlinks,linkcolor=blue,citecolor=blue,urlcolor=blue,
linktocpage=true]{hyperref}
\usepackage{srcltx}
\usepackage[english]{babel}

\numberwithin{equation}{section}
\setinitialpage{1}
\renewcommand{\qedsymbol}{$\Box$}

\title{Computing minimal interpolants in $C^{1,1}(\mathbb{R}^d)$}

\author[A. Herbert-Voss, M. Hirn and F. McCollum]{Ariel
  Herbert-Voss, Matthew J. Hirn\thanks{Corresponding author} and
  Frederick McCollum}

\address[ariel\_herbertvoss@g.harvard.edu]{{\sc Ariel Herbert-Voss}:
  Harvard University, School of Engineering and Applied Sciences, 29
  Oxford Street, Cambridge, Massachusetts 02138, USA}
\address[mhirn@msu.edu]{{\sc Matthew J. Hirn}: Michigan State
  University, Department of Computational Mathematics, Science \&
  Engineering and Department of Mathematics, 619 Red Cedar Road, East
  Lansing, Michigan 48824, USA}
\address[frederick.mccollum@nyu.edu]{{\sc Frederick McCollum}: New York
  University, Courant Institute of Mathematical Sciences, 251 Mercer
  Street, New York, New York, 10012, USA}

\amsclassification[]{26B35, 41A05, 41A58, 41A63, 52A41, 65D05}


\usepackage{amsmath, amsfonts, amssymb, amstext}
\usepackage{enumitem}
\usepackage{graphicx}
\usepackage{subfigure}
\usepackage{stmaryrd}
\usepackage{color}
\usepackage{multirow}
\usepackage{appendix}
\usepackage{alltt}
\usepackage{mdframed}

\newcommand{\cell}{\mathrm{cell}}

\newcommand{\diam}{\mathrm{diam}}
\newcommand{\distset}{\mathrm{dist}}
\newcommand{\DT}{\mathrm{DT}}

\newcommand{\face}{\mathrm{face}}
\newcommand{\K}{\mathcal{K}}
\newcommand{\Lip}{\mathrm{Lip}}
\newcommand{\pow}{\mathrm{pow}}
\newcommand{\PD}{\mathrm{PD}}
\newcommand{\PP}{\mathcal{P}}
\newcommand{\R}{\mathbb{R}}
\newcommand{\T}{\mathcal{T}}
\newcommand{\W}{\mathcal{W}}
\newcommand{\Z}{\mathbb{Z}}

\theoremstyle{theorem}
\newtheorem{theorem}{Theorem}[section]
\newtheorem{proposition}[theorem]{Proposition}
\newtheorem{lemma}[theorem]{Lemma}
\newtheorem{corollary}[theorem]{Corollary}

\theoremstyle{remark}
\newtheorem{remark}[theorem]{Remark}

\begin{document}

\renewcommand{\thefootnote}{\fnsymbol{footnote}}

\maketitle

\begin{abstract}
We consider the following interpolation problem. Suppose one is given
a finite set $E \subset \R^d$, a function $f: E \rightarrow \R$, and
possibly the gradients of $f$ at the points of $E$. We want to
interpolate the given information with a function $F \in
C^{1,1}(\R^d)$ with the minimum possible value of $\Lip (\nabla
F)$. We present practical, efficient algorithms for constructing an $F$ such that
$\Lip(\nabla F)$ is minimal, or for less computational effort, within
a small dimensionless constant of being minimal.\\
\\
October 24, 2016 \\
\\
\href{http://arxiv.org/abs/1411.5668}{arXiv:1411.5668}
\end{abstract}

\thispagestyle{plain}

\section{Introduction}

We consider the problem of computing interpolants in
$C^{1,1}(\R^d)$, that is the space of functions whose
derivatives are Lipschitz:
\begin{align*}
C^{1,1}(\R^d) &= \{ g: \R^d \rightarrow \R \mid \Lip ( \nabla g) <
\infty \}, \\
\Lip ( \nabla g) &= \sup_{\substack{x, y \in \R^d \\ x
    \neq  y}} \frac{ | \nabla g(x) - \nabla g(y) |}{|x - y|},
\end{align*}
where $| \cdot |$ denotes the Euclidean norm. Analogous to interpolation by
Lipschitz functions, in which one wishes to minimize the Lipschitz
constant of the interpolant, here we aim to minimize the Lipschitz
constant of the gradient of the interpolant. We consider two closely related problems,
the first of which is the following: \\

\noindent\underline{\textsc{Jet interpolation problem:}} One is given
a finite set of points $E \subset \R^d$, and at each point $a \in E$,
a function value $f_a \in \R$ and a gradient $D_af \in \R^d$ are
specified. Compute an interpolating function $F \in C^{1,1}(\R^d)$ such that:
\begin{enumerate}
\item
$F(a) = f_a$ and $\nabla F(a) = D_af$ for each $a \in E$.
\item
Amongst all such interpolants satisfying the previous condition, the
value of $\Lip(\nabla F)$ is minimal. \\
\end{enumerate}

The \textsc{Jet interpolation problem} is a computational version of
Whitney's Extension Theorem \cite{whitney:analyticExtensions1934}. Whitney's
Extension Theorem is a partial converse to Taylor's Theorem. Given a
closed set $E \subset \R^d$ (not necessarily finite) and an
$m^{\text{th}}$ degree polynomial at each point of $E$, Whitney's
Extension Theorem states that if the collection of polynomials
satisfies certain compatibility conditions, then there exists a
function $F \in C^m(\R^d)$ such that for each point $a \in E$ the
$m^{\text{th}}$ order Taylor expansion of $F$ at $a$ agrees with the
polynomial specified at that point. A similar result can be stated for
$C^{m-1,1}(\R^d)$ with $(m-1)^{\text{st}}$ degree polynomials given at each
point of $E$. In the case of $C^{1,1}(\R^d)$, the polynomials are
defined by the specified function and gradient information:
\begin{equation*}
P_a(x) = f_a + D_af \cdot (x-a), \quad a \in E, \enspace x \in \R^d.
\end{equation*}
Letting $\PP$ denote the space of first order polynomials, the map
\begin{align*}
P: \R^d &\rightarrow \PP, \\
a &\mapsto P_a,
\end{align*}
is called a \emph{1-field} (or a Whitney field). For a function $F \in
C^{1,1}(\R^d)$, the first order Taylor expansions of $F$ are elements
of $\PP$. Such expansions are called \emph{jets}, and are defined as:
\begin{equation*}
J_aF(x) = F(a) + \nabla F(a) \cdot (x-a), \quad a,x \in \R^d.
\end{equation*}
Whitney's Extension Theorem for $C^{1,1}(\R^d)$ can then be stated as
follows.

\begin{theorem}[Whitney's Extension Theorem for
  $C^{1,1}(\R^d)$] \label{thm: WET C11}
Let $E \subset \R^d$ be closed and let $P: E \rightarrow \PP$ be a
1-field with domain $E$. If there exists a constant $M < \infty$
such that
\begin{enumerate}
\setcounter{enumi}{-1}
\renewcommand{\theenumi}{$(W_{\arabic{enumi}})$}
\renewcommand{\labelenumi}{\theenumi}
\item\label{W0}
$|P_a(a) - P_b(a)| \leq M |a-b|^2$, for all $a,b \in E$,
\item\label{W1}
$|\frac{\partial P_a}{\partial x_i}(a) - \frac{\partial P_b}{\partial
  x_i}(a)| \leq M |a-b|$, for all $a,b \in E$ and $i = 1,\ldots, d$,
\end{enumerate}
then there exists an extension $F \in C^{1,1}(\R^d)$ such that $J_aF =
P_a$ for all $a \in E$.
\end{theorem}

When the set $E$ is finite, the compatibility conditions of Whitney's
Extension Theorem are automatically satisfied. On the other hand, the
theorem cannot be used to derive the minimal value of $\Lip(\nabla
F)$. Denote this value as:
\begin{equation*}
\|P\|_{C^{1,1}(E)} = \inf \{ \Lip(\nabla \widetilde{F}) \mid J_a \widetilde{F} = P_a \text{ for
  all } a \in E\}.
\end{equation*}
Indeed, if one were to take the infimum over all possible $M$
satisfying \ref{W0} and \ref{W1},
the resulting value would only be within a constant $C(d)$ of
$\|P\|_{C^{1,1}(E)}$. A recent paper by Le Gruyer
\cite{legruyer:minLipschitzExt2009} solves this problem in closed
form by defining a functional $\Gamma^1$ such that $\Gamma^1(P;E) =
\|P\|_{C^{1,1}(E)}$. This is the first ingredient in our solution to
the \textsc{Jet interpolation problem}. The second is a result of Wells
\cite{wells:diffFuncLipDeriv1973}, which gives a construction of an
interpolant $F \in C^{1,1}(\R^d)$ with $\Lip(\nabla F) = M$ for a
specified value of $M$ satisfying certain conditions. It is easy to show
that $M = \Gamma^1(P;E)$ satisfies the conditions, and thus one can
combine the two results to obtain a minimal interpolant. The
construction of Wells is not simple, however, and must be adapted to
certain data structures implementable on a computer. Thus our third
ingredient is a collection of algorithms and data structures from
computational geometry that compute and encode the results of Le Gruyer
and Wells. 

The second problem we consider is when only the function values are
specified: \\

\noindent\underline{\textsc{Function interpolation problem}:} One is
given a finite set of points $E \subset \R^d$ and a function $f: E
\rightarrow \R$. Compute an interpolating function $F \in
C^{1,1}(\R^d)$ such that:
\begin{enumerate}
\item
$F(a) = f(a)$ for each $a \in E$.
\item
Amongst all such interpolants satisfying the previous condition, the value
of $\Lip(\nabla F)$ is minimal. \\
\end{enumerate}

The \textsc{Function interpolation problem} is harder than the
\textsc{Jet interpolation problem} since the space of functions
satisfying $F(a) = f(a)$ is larger than the space of functions
satisfying $J_aF = P_a$. However, the minimal value of $\Lip(\nabla
F)$ can be specified using the functional $\Gamma^1$. Indeed, let
$\|f\|_{C^{1,1}(E)}$ denote the minimal value of $\Lip(\nabla F)$,
where
\begin{equation*}
\|f\|_{C^{1,1}(E)} = \inf \{ \Lip(\nabla \widetilde{F}) \mid \widetilde{F}(a) = f(a) \text{
  for all } a \in E\}.
\end{equation*}
For functions $f$, define the $\Gamma^1$ functional as:
\begin{equation*}
\Gamma^1(f;E) = \inf \{ \Gamma^1(P;E) \mid P_a(a) = f(a) \text{ for
  all } a \in E\}.
\end{equation*}
Then, as is shown in \cite{legruyer:minLipschitzExt2009},
$\Gamma^1(f;E) = \|f\|_{C^{1,1}(E)}$. The functional $\Gamma^1$ is
convex, and thus $\Gamma^1(f;E)$ can be computed using convex
programming. Additionally, the minimizing 1-field can be
outputted. Then one can use the remainder of the \textsc{Jet
  interpolation problem} algorithm to solve the \textsc{Function
  interpolation problem}.

The purpose of this paper is twofold, with one aspect being the \emph{theoretical} efficiency of our
algorithms, but the other, equally important aspect, being the
\emph{practicality} of our algorithms. Indeed, the goal is to balance
the two; sometimes this results in trading theoretical efficiency for algorithms that
can be implemented and run a computer, while in other cases
we prove new theoretical results that are of practical interest. 

On the theoretical side, we assume that our computer is able to
work with exact real numbers. We ignore roundoff, overflow, and
underflow errors, and suppose that an exact real number can be stored
at each memory address. Additionally, we suppose that it takes one
machine operation to add, subtract, multiply, or divide two real
numbers $x$ and $y$, or to compare them (i.e., decide whether $x < y$,
$x > y$, or $x = y$).

From a more practical perspective, we do not always implement
algorithms with optimal theoretical worst case guarantees in terms
of complexity, opting to choose an alternative that works better in
practice. Additionally, at various stages of our interpolation algorithm
we may give multiple options for computing the next step. The difference
between the options might depend on $N$ and $d$, or one might be
more stable than another in certain situations. In Section \ref{sec:
  numerical simulations} we give a summary of numerical simulations of
the algorithm running on a computer. 

The \emph{work} of an algorithm is the number of machine operations
needed to carry it out, and the \emph{storage} of an algorithm is the
number of random access memory addresses required.

In order to analyze the complexity of the algorithm, we break it into
three main components: the aforementioned storage, the \emph{one time
  work}, and the \emph{query work}. The one time work consists of the
following: given the set $E$ and either the 1-field $P$ or the function $f$, the algorithm
performs a certain amount of preprocessing, after which it is ready to
accept queries from the user. The number of computations needed for
this preprocessing stage is the one time work. The algorithm
additionally outputs $\Lip(\nabla F)$ at the end of this stage. Once the one
time work is complete, the user inputs a query point $x \in \R^d$, and the
algorithm returns $J_xF$ (i.e. $F(x)$ along with $\nabla F(x)$). The
amount of work for each query is the query work. 

Recently, there has been an interest in algorithmic results
related to Whitney's Extension Theorem. One can recast the
\textsc{Function interpolation problem} in terms of interpolants $F
\in C^m(\R^d)$, with the goal to minimize an appropriate $C^m$
norm. Suppose that $\#(E) = N$. In \cite{fefferman:fittingDataI}, an
algorithm is presented which computes a number $M$ such that $M$ has
the same order magnitude as $\|f\|_{C^m(E)}$, that is $c(m,d) M \leq
\|f\|_{C^m(E)} \leq C(m,d) M$. The algorithm
requires $O(N\log N)$ work and $O(N)$ storage. In a second, 
companion paper \cite{fefferman:fittingDataII}, an additional
algorithm is presented which computes a function 
$F \in C^m(\R^d)$ such that $F$ interpolates
the given function and $c(m,d) \|F\|_{C^m(\R^d)} \leq \|f\|_{C^m(E)}
\leq C(m,d) \|F\|_{C^m(\R^d)}$. The one time work of the algorithm requires
$O(N\log N)$ operations, the query work requires
$O(\log N)$ operations, and the storage never exceeds
$O(N)$. Analogous results describing a new algorithm for fitting a
Sobolev function to data are presented in \cite{fefferman:sobolevDataI}.

In related work \cite{fefferman:interLinProg2011}, the task of
computing an $F$ such that $\|F\|_{C^m(\R^d)}$
is within a factor of $1 + \epsilon$ of $\|f\|_{C^m(E)}$ is considered. A linear
programming problem is devised which solves the problem. The number of
linear constraints grows linearly in $N$ and as
$O(\epsilon^{-\frac{3}{2}d})$ in $\epsilon$.

In terms of the efficiency with regards to $N$, these algorithms are
optimal. However, the dimension dependent constants can grow
exponentially with $d$. Additionally, while the algorithms are
beautiful, they are also intricate. Thus, from a practical
perspective, they are not likely to be implemented on a
computer and used in applications. 

Our algorithm, while not always reaching the optimal theoretical
complexity guarantees of the previously mentioned algorithms, is to
the best of our knowledge the first of its type to be implemented on a
computer and be practical for certain tasks. The key features of our $C^{1,1}$ algorithm are:
\begin{itemize}

\item
We can compute $\|P\|_{C^{1,1}(E)}$ precisely in $O(N^2)$ work, or to
within a dimensionless constant (approximately $20$) of
$\|P\|_{C^{1,1}(E)}$ using $O(N\log N)$ work.

\item
We can compute $\|f\|_{C^{1,1}(E)}$ to within a dimensionless constant
(again approximately $20$) plus an arbitrarily small additive slack
$\epsilon$ using interior point methods from convex optimization. The
number of iterations is sublinear,
$O((N \log N)^{1/2} \log(N/\epsilon))$, and the cost per iteration is $O(N^3)$,
although sparsity considerations may improve this computational time
in practice. 

\item
For the one time work outside of computing $\|P\|_{C^{1,1}(E)}$ or
$\|f\|_{C^{1,1}(E)}$, the efficiency of the algorithm is tied to the
complexity of computing and storing a convex hull. Computing a convex hull
requires $O(N\log N + N^{\lceil d/2 \rceil})$ operations, and its
storage is $O(N^{\lceil d/2 \rceil})$. Subsequently, the work of our
algorithm is $O(N^{d+2})$ and the storage is $O(N^{\lceil d/2 \rceil +
  1})$.

\item
With some additional one time work and under suitable conditions, the
query work requires only $O(\log N)$ operations.

\item
A version of the algorithm that runs from start to finish has been
implemented in $\text{MATLAB}^{\circledR}$ and can be run on a laptop for small problems,
and is practical on a server for larger problems. The
complete code can be downloaded at:
\begin{equation*}
\href{https://github.com/matthew-hirn/C-1-1-Interpolation}{\texttt{https://github.com/matthew-hirn/C-1-1-Interpolation}}
\end{equation*}
The code is easy to use and not difficult to edit. Throughout the
paper we highlight which parts of the 
algorithm have been implemented, and discuss the potential benefits of
the parts that have not been implemented.

\item
The interpolant we compute is a generalized absolutely minimal
Lipschitz extension (AMLE), as defined in \cite{hirn:quasiAMLE2012}.

\end{itemize}

Finally, the choice of the space $C^{1,1}(\R^d)$ is natural in several
ways beyond the existence of the results of Le Gruyer and
Wells. Standard Lipschitz extensions in $C^{0,1}(\R^d)$ and absolutely minimal Lipschitz
extensions have applications in computer
science \cite{kyng:lipLearnGraphs2015}, partial differential equations
\cite{aronsson:sandpiles1996}, and image processing
\cite{caselles:axiomaticImage1998}, among others. The next non-trivial
space to consider beyond $C^{0,1}(\R^d)$ is $C^{1,1}(\R^d)$. In fact,
for $d=1$, the more general space $C^{m-1,1}(\R)$ was originally considered by Favard
\cite{favard:interpolation1940} and later by Glaeser
\cite{glaeser:prolongExtremal1973}. The solution $F$ of the
\textsc{Jet interpolation problem} for this space is a spline $F \in
C^{m-1, 1}(\R)$ made up of $C^m$-smooth pieces with at most $m-1$
knots. In cubic spline interpolation the spline either minimizes
the integral of the curvature $\kappa$ \cite{glass:splines1966},
\begin{equation*}
\int \kappa(x) \, dx = \int \frac{|F^{(2)}(x)|^2}{(1 + F^{(1)}(x)^2)^{5/2}} \, dx,
\end{equation*}
or a simpler energy \cite{holladay:splines1957} such as:
\begin{equation} \label{eqn: spline l2 energy}
\|F^{(2)}\|_{L^2(\R)}^2 = \int |F^{(2)}(x)|^2 \, dx.
\end{equation}
In the setting of $C^{m-1,1}(\R)$, though, the optimal splines minimize
the energy:
\begin{equation*}
\Lip(F^{(m-1)}) = \sup_{x \in \R} |F^{(m)}(x)| = \|F^{(m)}\|_{L^{\infty}(\R)}.
\end{equation*}
In light of \eqref{eqn: spline l2 energy}, we see that in the $C^{m-1,1}(\R)$
setting we have replaced the $L^2$ energy with an $L^{\infty}$
energy. In $d$-dimensions, a similar identity holds for $C^{1,1}(\R^d)$:
\begin{equation*}
\Lip(\nabla F) = \sup_{x \in \R^d} \| \nabla^2 F(x) \|_{\mathrm{op}},
\end{equation*}
where $\| \cdot \|_{\mathrm{op}}$ is the operator norm. The interpolants we compute
are piecewise quadratic, and thus are $d$-dimensional analogues to the
quadratic splines of Glaeser. The ``knots'' in higher dimensions
correspond to $(d-1)$-dimensional facets (for example, line segments in
$\R^2$), along which there is a discontinuity in the second order
partial derivatives of $F$.

In the work of Glaeser as well as many cubic spline interpolating schemes, the
ordering of the real line allows one to reduce the $N$ point
interpolation problem to a two point interpolation problem. When one
transitions to $\R^d$ and asks for $d$-dimensional interpolants,
however, the lack of natural ordering is problematic. This is where
the work of Wells and Le Gruyer come into play, giving us a roadmap to
navigate the higher dimensional Euclidean space.

While we have not applied our algorithm on real data, it would seem
that the algorithm could be useful for various applications. For
example, it could be used to aid in the design of experiments in
applied physics and chemistry. Suppose a scientist wants to conduct a
costly experiment in which he must deposit a thin film of $\text{SiO}_2$ in a
special tool that has plasma in it. The success of this experiment
depends on several factors, such as the pressure in the chamber,
the temperature of the substrate, the voltage of the plasma, and the ratios of
the gases involved. He wants to find the optimal conditions for performing
the experiment. He knows that the voltage is a smooth function of the
other parameters, but it is difficult to measure. Consequently, he can
only measure it for a few different combinations of initial
conditions. He varies each parameter slightly while holding the others
constant to find the rate of change of the voltage with respect to that
parameter. Now he has data points (configurations of the parameters),
function values (measured voltages), and partial derivatives. Using
our interpolation algorithm, it is possible to compute a good estimate
of the voltage for any configuration of the parameters and thereby
determine the optimal conditions for the experiment.

The remainder of this paper is organized as follows. In Section
\ref{sec: background} we give the relevant background information
regarding the results of Wells and Le Gruyer, as well as the pertinent
material from computational geometry. In Section \ref{sec: overview of
  algorithm} we present an overview of the algorithm, while in
Sections \ref{sec: compute gamma1}, \ref{sec: one time work}, and
\ref{sec: query work} we fill in the details. In particular, we
describe efficient algorithms for computing $\Gamma^1$ in Section
\ref{sec: compute gamma1}, the remainder of the one time work is
detailed in Section \ref{sec: one time work}, and in Section \ref{sec:
  query work} we present algorithms for the query work. Section
\ref{sec: numerical simulations} describes numerical simulations of
the algorithm and its resulting performance. Appendix \ref{sec:
  convex optimization} reviews some standard textbook concepts from
convex optimization, which can be found in
\cite{boyd:convexOptimization2004}. 

\section{Background} \label{sec: background}

In this section we review the results of Wells and Le Gruyer, and
go over the relevant material from computational geometry.

\subsection{Wells: Constructing the interpolant} \label{sec: wells}

In \cite{wells:diffFuncLipDeriv1973} Wells describes a construction of
an interpolant $F \in C^{1,1}(\R^d)$ with specified semi-norm
$\Lip(\nabla F) = M$. Our algorithm will be based on this
construction, which we review here.

The inputs are the set $E \subset \R^d$, the 1-field $P: E
\rightarrow \PP$, which consists of the specified function values $\{f_a\}_{a
  \in E} \subset \R$ and gradients $\{D_af\}_{a \in E} \subset
\R^d$, and the value $M$. In order for Wells' construction to hold, the
following condition must be satisfied:
\begin{equation} \label{eqn: Wells condition}
f_b \leq f_a + \frac{1}{2}(D_af + D_bf) \cdot (b-a) +
\frac{M}{4}|b-a|^2 - \frac{1}{4M} |D_af - D_bf|^2, \enspace \forall \, 
a,b \in E.
\end{equation}

For each point $a \in E$, define a shifted point $\tilde{a}$:
\begin{equation*}
\tilde{a} = a - \frac{D_af}{M}, \quad a \in E.
\end{equation*}
Additionally, to each point $a \in E$ Wells associates a type of distance function
$d_a: \R^d \rightarrow \R$ from $\R^d$ to that point:
\begin{equation} \label{eqn: wells da dist}
d_a(x) = f_a - \frac{1}{2M} |D_af|^2 + \frac{M}{4} |x - \tilde{a}|^2,
\quad a \in E, \enspace x \in \R^d.
\end{equation}
For any subset $S \subset E$ define $d_S: \R^d \rightarrow \R$ as 
\begin{equation*}
d_S(x) = \min_{a \in S} d_a(x), \quad x \in \R^d.
\end{equation*}
Using the shifted points and the distance functions, Wells associates
to every subset $S \subset E$ several new sets:
\begin{align*}
\widetilde{S} &= \{ \tilde{a} \mid a \in S \}, \\
S_H &= \text{the smallest affine space containing } \widetilde{S}, \\
\widehat{S} &= \text{the convex hull of } \widetilde{S}, \\
S_E &= \{x \in \R^d \mid d_a(x) = d_b(x), \enspace \forall \, a,b \in
S\}, \\
S_{\ast} &= \{x \in \R^d \mid d_a(x) = d_b(x) \leq d_c(x), \enspace
\forall \, a,b \in S, \enspace c \in E\}, \\
S_C &= S_H \cap S_E.
\end{align*}
Wells also defines a set of special subsets $S \subset E$:
\begin{equation} \label{eqn: wells special sets K}
\K = \{S \subset E \mid \exists \, x \in S_{\ast} \text{ such that }
d_S(x) < d_{E \setminus S}(x) \}.
\end{equation}
Note that when $S \in \K$, $S_E \neq \emptyset$, $\dim S_E + \dim S_H
= d$, and $S_H \perp S_E$; therefore $S_C$ is a single point in
$\R^d$. Using the subsets contained in
$\K$, Wells defines a new collection of sets $\{T_S\}_{S \in \K}$,
\begin{equation} \label{eqn: TS definition}
T_S = \frac{1}{2}(\widehat{S} + S_{\ast}) = \left\{ \frac{1}{2}(y + z)
  \mid y \in \widehat{S}, \enspace z \in S_{\ast} \right\}, \quad S \in \K.
\end{equation}
The collection $\{T_S\}_{S \in \K}$ forms a covering of $\R^d$ in
which the regions of overlap have Lebesgue
measure zero. On each set $T_S$, Wells defines a function $F_S: T_S
\rightarrow \R$, which is a local piece of the final interpolant:
\begin{equation*}
F_S(x) = d_S(S_C) + \frac{M}{2} \distset(x,S_H)^2 - \frac{M}{2} \distset(x,S_E)^2,
\quad x \in T_S, \enspace S \in \K,
\end{equation*}
where for any two sets $U,V \subset \R^d$
\begin{equation*}
\distset(U,V) = \inf_{\substack{x \in U\\y \in V}} |x-y|.
\end{equation*}
The final function $F: \R^d \rightarrow \R$ is defined as:
\begin{equation} \label{eqn: Wells interpolant}
F(x) = F_S(x), \quad \text{if } x \in T_S.
\end{equation}
If $T_{S} \cap T_{S'} \neq \emptyset$, then $F_S$ and
$F_{S'}$ as well as $\nabla F_S$ and $\nabla F_{S'}$ agree on $T_{S}
\cap T_{S'}$, so $F$ is well defined and $F \in
C^{1,1}(\R^d)$. Additionally, the gradient of $F_S$ has a simple
analytic form, given by:
\begin{equation*}
\nabla F_S(x) = \frac{M}{2}(z-y), \quad x=\frac{1}{2}(y+z), \enspace y
\in \widehat{S}, \enspace z \in S_{\ast}.
\end{equation*}
Finally, the function $F$ interpolates the data and has the
prescribed semi-norm:
\begin{theorem}[Wells, {\cite[Section 4,
  Theorem 1]{wells:diffFuncLipDeriv1973}}]\label{thm: wells}
Given a finite set $E \subset \R^d$, a 1-field $P: E \rightarrow
\PP$, and a constant $M$ satisfying \eqref{eqn: Wells condition}, the
function $F: \R^d \rightarrow \R$ defined by \eqref{eqn: Wells
  interpolant} is in $C^{1,1}(\R^d)$ and additionally:
\begin{enumerate}
\item
$J_aF = P_a$ for all $a \in E$,
\item
$\Lip(\nabla F) = M$.
\end{enumerate}
\end{theorem}

\subsection{Le Gruyer: The minimal value of $\Lip (\nabla F)$}

While the result of Wells gives a construction for an interpolant with
prescribed semi-norm $M$, it does not explicitly give the minimum possible
value of $M$. Recall that this minimum value is defined for $1$-fields
as:
\begin{equation*}
\|P\|_{C^{1,1}(E)} = \inf \{ \Lip(\nabla \widetilde{F}) \mid J_a \widetilde{F} = P_a \text{ for
  all } a \in E\},
\end{equation*}
and for functions as:
\begin{equation*}
\|f\|_{C^{1,1}(E)} = \inf \{ \Lip(\nabla \widetilde{F}) \mid \widetilde{F}(a) = f(a) \text{
  for all } a \in E\}.
\end{equation*}
As in \cite{legruyer:minLipschitzExt2009}, define the functional $\Gamma^1$ as:
\begin{equation} \label{eqn: Gamma1 orig}
\Gamma^1(P;E) = 2 \sup_{x \in \R^d} \max_{\substack{a,b \in E\\a \neq
    b}} \frac{|P_a(x) - P_b(x)|}{|a-x|^2 + |b-x|^2}.
\end{equation}
Recall that for functions we defined $\Gamma^1(f;E)$ as:
\begin{equation*}
\Gamma^1(f;E) = \inf \{ \Gamma^1(P;E) \mid P_a(a) = f(a) \text{ for
  all } a \in E\}.
\end{equation*}
The following two results show that $\Gamma^1$ is equivalent to $\|
\cdot \|_{C^{1,1}(E)}$.

\begin{theorem}[Le Gruyer,
  {\cite[Theorem 1.1]{legruyer:minLipschitzExt2009}}]\label{thm: le gruyer}
Given a set $E \subset \R^d$ and a 1-field $P: E \rightarrow
\PP$, 
\begin{equation*}
\Gamma^1(P;E) = \|P\|_{C^{1,1}(E)}.
\end{equation*}
\end{theorem}

\begin{corollary}[Le Gruyer,
  {\cite[Theorem 3.2]{legruyer:minLipschitzExt2009}}]\label{cor: le gruyer}
Given a set $E \subset \R^d$ and a function $f: E \rightarrow \R$, 
\begin{equation*}
\Gamma^1(f;E) = \|f\|_{C^{1,1}(E)}.
\end{equation*}
\end{corollary}

The functional $\Gamma^1(P;E)$ has an alternate form which
will prove to be more useful than \eqref{eqn: Gamma1 orig} from a
computational perspective. Define two additional functionals:
\begin{align*}
A(P;a,b) &= \frac{|P_a(a) - P_b(a) + P_a(b) - P_b(b)|}{|a-b|^2}, \quad
a,b \in E\\
B(P;a,b) &= \frac{|\nabla P_a(a) - \nabla P_b(a)|}{|a-b|} =
\frac{|D_af - D_bf|}{|a-b|}, \quad a,b \in E.
\end{align*}
Then one can show \cite[Proposition 2.2]{legruyer:minLipschitzExt2009}:
\begin{equation} \label{eqn: Gamma1 alternate}
\Gamma^1(P;E) = \max_{\substack{a,b \in E\\a \neq b}} \sqrt{ A(P;a,b)^2
  + B(P;a,b)^2 } + A(P;a,b).
\end{equation}
This alternate form removes the supremum and reduces the work of
computing $\Gamma^1(P;E)$ to $O(N^2)$. Additionally, using \eqref{eqn:
  Gamma1 alternate} it is not hard to show that $M = \Gamma^1(P;E)$
satisfies the Wells condition \eqref{eqn: Wells condition}. Thus
combining Wells' Theorem \ref{thm: wells} and Le Gruyer's Theorem
\ref{thm: le gruyer} we arrive at a minimal interpolant for the
\textsc{Jet interpolation problem}. One can utilize Corollary
\ref{cor: le gruyer} and Theorem \ref{thm: wells} to obtain a solution
for the \textsc{Function interpolation problem}, assuming that when one
solves for $\Gamma^1(f;E)$ the minimizing 1-field is outputted as
well. 

In fact, by a recent result contained in
\cite{legruyer:LipOneFields2014}, the interpolant $F$ of Wells with $M
= \Gamma^1(P;E)$ is a
generalized absolutely minimal Lipschitz extension (AMLE) according to
the definition presented in \cite{hirn:quasiAMLE2012}. To understand
this statement, first note that the interpolant $F$ defines a
1-field through its jets, namely:
\begin{align*}
P^{(F)}: \R^d &\rightarrow \PP, \\
a &\mapsto J_aF.
\end{align*}
The functional $\Gamma^1$ can be thought of as the Lipschitz constant
for 1-fields. Indeed, aside from the main result that $\Gamma^1$
is equal to the minimum value of $\Lip(\nabla \widetilde{F})$ as
$\widetilde{F}$ ranges over all interpolants, one can show additionally that,
\begin{equation*}
\Gamma^1(P^{(F)};\R^d) = \Gamma^1(P;E).
\end{equation*}
Thus the 1-field $P^{(F)}$ extends $P$ while preserving
$\Gamma^1$. Note that this is analogous to the standard Lipschitz
extension problem between Hilbert spaces, for which it is known that any
Lipschitz function mapping a subset of Hilbert space to another
Hilbert space can be extended while preserving the Lipschitz
constant \cite{kirszbraun:lipschitzTransformations1934}. For real
valued Lipschitz extensions, the notion of an AMLE goes back to
Aronsson \cite{aronsson:minSupFI1965,aronsson:minSupFII1966,aronsson:AMLE1967}
and has been studied extensively due to its relationship to partial
differential equations \cite{jensen:uniqueAMLE1993}, stochastic games
\cite{peres:tugOfWar2009}, and applications in applied mathematics
\cite{aronsson:sandpile1965, aronsson:sandpiles1996,
  caselles:axiomaticImage1998, almansa:elevationModelAMLE2002}. An
AMLE is the locally best Lipschitz extension. The formal definition can be extended to
other functionals such as $\Gamma^1$, where in this case we say an
extension $Q: \R^d \rightarrow \PP$ of $P$ is an AMLE if
\begin{enumerate}
\item
$\Gamma^1(Q;\R^d) = \Gamma^1(P;E)$,
\item
For every open subset $V \subset \R^d \setminus E$,
\begin{equation*}
\Gamma^1(Q;V) = \Gamma^1(Q;\partial V).
\end{equation*}
\end{enumerate}
A result in \cite{legruyer:LipOneFields2014} states that $P^{(F)}$
is an AMLE when $E$ is finite. Thus the interpolant that our algorithm computes is an
AMLE for $C^{1,1}(\R^d)$. Given the interest in classical AMLEs,
having an algorithm to compute them in the $C^{1,1}(\R^d)$ case has
the potential to be of use in suitable applications. 

\subsection{Computational geometry}

We now review the relevant material from computational geometry.

\subsubsection{Well separated pairs decomposition} \label{sec: wspd}

The following is relevant for computing approximations of
$\Gamma^1$ when the number of points $N$ is large. The well
separated pairs decomposition was first introduced by
Callahan and Kosaraju in \cite{callahan:wspd1995}; we shall make use
of a modified version that was described in detail in
\cite{fefferman:fittingDataI}.

First, let $U,V \subset \R^d$ and recall the definitions of the diameter of a set
and the distance between two sets:
\begin{equation*}
\diam(U) = \sup_{\substack{x,y \in U \\ x \neq y}} |x-y|,
\qquad \distset(U,V) = \inf_{\substack{x \in U \\ y \in V}} |x-y|.
\end{equation*}
For $\varepsilon > 0$, two sets $U, V \subset \R^d$ are
$\varepsilon$\emph{-separated} if
\begin{equation*}
\max \{\diam(U), \diam(V)\} < \varepsilon \cdot \distset(U,V).
\end{equation*}

We follow the construction detailed by Fefferman and Klartag in
\cite{fefferman:fittingDataI}. Let $\T$ be a collection of subsets of
$E$. For any $\Lambda \subset
\T$, set
\begin{equation*}
\cup \Lambda = \bigcup_{S \in \Lambda} S = \{x \mid x \in S
\text{ for some } S \in \Lambda \}.
\end{equation*}
Let $\W$ be a set of pairs $(\Lambda_1, \Lambda_2)$ where $\Lambda_1,
\Lambda_2 \subset \T$. For any $\varepsilon > 0$, the pair $(\T, \W)$
is an \emph{$\varepsilon$-well separated pairs decomposition} or
\emph{$\varepsilon$-WSPD} for short if the following properties hold:
\begin{enumerate}

\item \label{item: F-K 1}
$\bigcup_{(\Lambda_1,\Lambda_2) \in \W} \cup \Lambda_1 \times \cup
\Lambda_2 = \{ (x,y) \in E \times E \mid x \neq y \}$.

\item \label{item: F-K 2}
If $(\Lambda_1, \Lambda_2), (\Lambda_1', \Lambda_2') \in \W$ are
distinct pairs, then $(\cup \Lambda_1 \times \cup \Lambda_2) \cap
(\cup \Lambda_1' \times \cup \Lambda_2') = \emptyset$.

\item
$\cup \Lambda_1$ and $\cup \Lambda_2$ are $\varepsilon$-separated for
any $(\Lambda_1, \Lambda_2) \in \W$.

\item
$\#(\T) < C(\varepsilon, d) N$ and $\#(\W) < C(\varepsilon, d) N$.

\end{enumerate}

As shown in \cite{fefferman:fittingDataI}, there is a data structure
representing $(\T,\W)$ that satisfies the following additional
properties as well:
\begin{enumerate}[resume]

\item
The amount of storage to hold the data structure is no more than $C
(\varepsilon, d) N$.

\item \label{item: F-K 6}
The following tasks require at most $C(\varepsilon, d) N \log N$ work and
$C(\varepsilon, d) N$ storage:
\begin{enumerate}

\item
Go over all $S \in \T$, and for each $S$ produce a list of elements in
$S$.

\item
Go over all $(\Lambda_1,\Lambda_2) \in W$, and for each $(\Lambda_1,
\Lambda_2)$ produce the elements (in $\T$) of $\Lambda_1$ and
$\Lambda_2$.

\item
Go over all $S \in \T$, and for each $S$ produce the list of all
$(\Lambda_1, \Lambda_2) \in W$ such that $S \in \Lambda_1$.

\item
Go over all $x \in E$, and for each $x \in E$ produce a list of $S \in
\T$ such that $x \in S$.

\end{enumerate}

\item \label{item: F-K 7}
As a result of property \ref{item: F-K 6}, the following properties also hold:
\begin{enumerate}[topsep=0pt]

\item
$\sum_{(\Lambda_1, \Lambda_2) \in \W} (\#(\Lambda_1) + \#(\Lambda_2)) <
C(\varepsilon,d) N\log N$.

\item
$\sum_{S \in \T} \#(S) < C(\varepsilon,d) N\log N$.

\end{enumerate}
\end{enumerate}
The next theorem gives bounds on the storage and work required to
compute an $\varepsilon$-WSPD.

\begin{theorem}[Fefferman and Klartag,
  {\cite[Theorem 5]{fefferman:fittingDataI}}] \label{thm: F-K WSPD}
There is an algorithm, whose inputs are the parameter $\varepsilon >
0$ and a set $E \subset \R^d$ with $\#(E) = N$, that outputs an
$\varepsilon$-WSPD $(\T,\W)$ of $E$ such that properties \ref{item:
  F-K 1},$\ldots$,\ref{item: F-K 7} hold. The algorithm requires no
more than $C(\varepsilon, d) N \log N$ work and
$C (\varepsilon, d) N$ storage.
\end{theorem}

\begin{remark} \label{rem: C(e,d)}
For $\varepsilon \leq 1/2$, the dimensional constant is $C (\varepsilon,
d) = (C \cdot \sqrt{d}/\varepsilon)^d$. Indeed, the algorithm of
Theorem \ref{thm: F-K WSPD} is built upon the WSPD algorithm originally
presented in \cite{callahan:wspd1995}. The algorithm of \cite{callahan:wspd1995} outputs
a similar WSPD $(T, W)$, in which $T$ is an unbalanced fair split
tree. The set $\T$ is a balanced binary tree, derived from
$T$. The height of $\T$ is no more than $\lceil \log_2 N
\rceil + 1$, and $\# (\T) < 2N$. Furthermore, the list of pairs $\W$ is in one-to-one
correspondence with $W$. Therefore the work and storage of both
algorithms is of the same order of magnitude. Examining
the proof of \cite[Lemma 4.2]{callahan:wspd1995} shows that the
number of pairs in $W$ is no more than $2(N-1)(3(\varepsilon^{-1}
\sqrt{d} + 2 \sqrt{d} + 1) + 2)^d)$, which is bounded from above by
$2N \cdot (10 \sqrt{d} / \varepsilon)^d$ when $\varepsilon \leq 1/2$.
\end{remark}

\subsubsection{Power diagrams, triangulations, and convex
  hulls} \label{sec: pd dt ch}

Now we switch to geometrical structures useful for computing the
interpolant $F$. A power diagram is a generalization of a Voronoi
diagram in which each
of the sites has an associated power function. Let $V \subset \R^d$ be a set of $n$
point sites. To each point $p \in V$, we associate a weight $w(p)$. The
power function $\pow: \R^d \times V \rightarrow \R$ measures the
distance from a point $x \in \R^d$ to a site $p \in V$ under the
influence of $w$. It is defined as:
\begin{equation*}
\pow(x,p) = |x-p|^2 - w(p).
\end{equation*}
The power cell of a point $p \in V$ is:
\begin{equation*}
\cell(p) = \{ x \in \R^d \mid \pow(x,p) \leq \pow(x,q), \enspace q
\in V \setminus \{p\} \}.
\end{equation*}
The set $\cell(p)$ can be empty; for generic sets $V$, when $\cell(p) \neq
\emptyset$, it is $d$-dimensional. Power cells are convex, but
possibly unbounded, polyhedra. The \emph{power diagram} of $V$,
denoted $\PD (V)$, is the convex polyhedral complex defined by these
cells.

The lower dimensional faces of $\PD (V)$ lie on the boundaries of the
power cells, which correspond to regions in $\R^d$ of equal power
between two or more sites. Define the face associated to the site $p$
and the set $U \subset V \setminus \{p\}$ as:
\begin{equation*}
\face (p, U) = \{ x \in \R^d \mid \pow (x, p) = \pow (x, q) \leq \pow (x,
r), \enspace q \in U, \enspace r \in V \setminus (U \cup \{p\}) \}.
\end{equation*}
Note that $\face (p, \emptyset) = \cell (p)$, and at times we will
refer to the power cell $\cell (p)$ as a $d$-dimensional face of $\PD
(V)$. Like the power cells, $\face (p, U)$ can be empty; when it is
not and the initial data $V$ is generic, $\dim \face (p, U) = d -
\#(U)$. All faces of $\PD (V)$ correspond to $\face (p, U)$ for some
site $p$ and set $U$. The $d-1$ dimensional faces are referred to as
facets, and the vertices (i.e., the zero dimensional faces) are called
\emph{power centers}. The latter are the points in $\R^d$ that are
equidistant to $d + 1$ points in $V$ relative to their power functions
(again assuming genericity of the initial data).

The geometric dual of $\PD (V)$ is a polyhedral cell complex $\DT (V)$
that satisfies the following property: For all $j = 0, \ldots, d$, there exists a bijective
mapping $\psi$ between the $j$-dimensional faces of $\PD (V)$ and the
$(d-j)$-dimensional faces of $\DT (V)$ such that if $\alpha, \beta$
are any two faces of $\PD (V)$, then $\alpha \subseteq \beta$ if and
only if $\psi (\beta) \subseteq \psi (\alpha)$. For generic initial data,
$\DT (V)$ is a triangulation, and when the power diagram is a Voronoi
diagram, $\DT (V)$ is a Delaunay triangulation \cite[Section
3.1.3]{aurenhammer:voronoi1991}. An example of a
power diagram and its dual  triangulation is given in Figure \ref{fig:
  PD and DT}.

\begin{figure}[h]
\center
\frame{\includegraphics[width=3.0in]{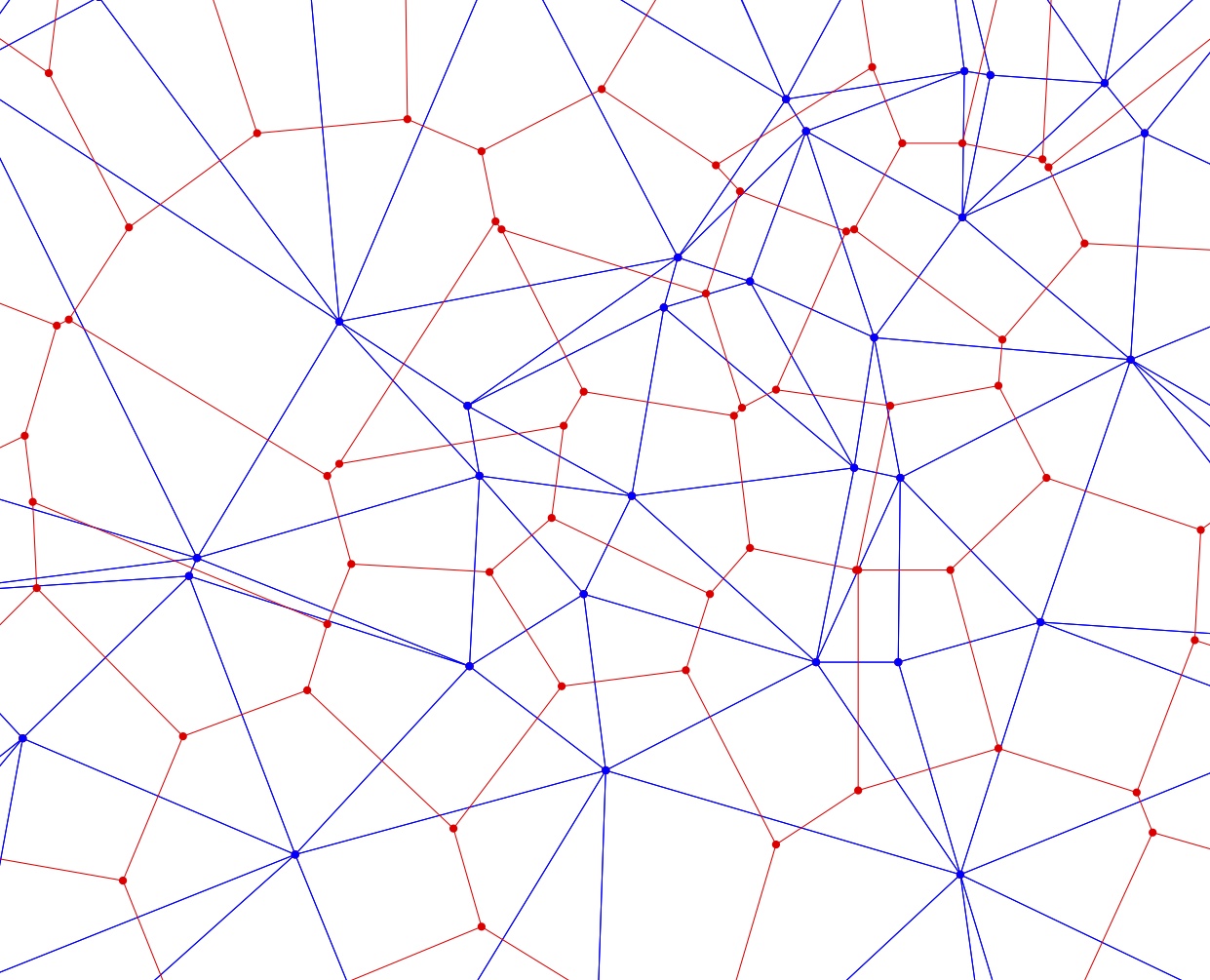}}
\caption{Power diagram in red with dual triangulation in
  blue. The original sites are the blue points, and the power centers
  are the red points.}
\label{fig: PD and DT}
\end{figure}

As shown in \cite[Section 3.1.3]{aurenhammer:voronoi1991}, there
is a close relationship between power diagrams and their dual
triangulations in $\R^d$, and convex hulls in $\R^{d+1}$. Indeed,
define the map $\lambda: V \rightarrow  \R^{d+1}$ as:
\begin{equation} \label{eqn: lambda lift}
\lambda(p) = (p, |p|^2 - w(p)).
\end{equation}
Consider the convex hull of the points $\{ \lambda(p) \mid p \in
V\}$. We can break it into two subsets, the lower hull and the upper
hull. We are interested in the lower hull, which consists of all
points that are visible from the point on the $x_{d+1}$ axis at
$-\infty$. Every $(d-j)$-dimensional face of the lower hull, for $j=0,
\ldots, d$, corresponds to a $(d-j)$-dimensional face of the 
triangulation $\DT(V)$. Furthermore, to obtain the faces of $\DT(V)$,
one simply makes an orthogonal projection of the lower hull back
onto $\R^d$. To obtain the power diagram $\PD(V)$, one uses the
duality of $\PD(V)$ to $\DT(V)$. An illustration of this process is
given in Figure \ref{fig: lifted CH for DT}. 

\begin{figure}[h]
\center
\subfigure[The lifted points]{
\includegraphics[width=2.0in]{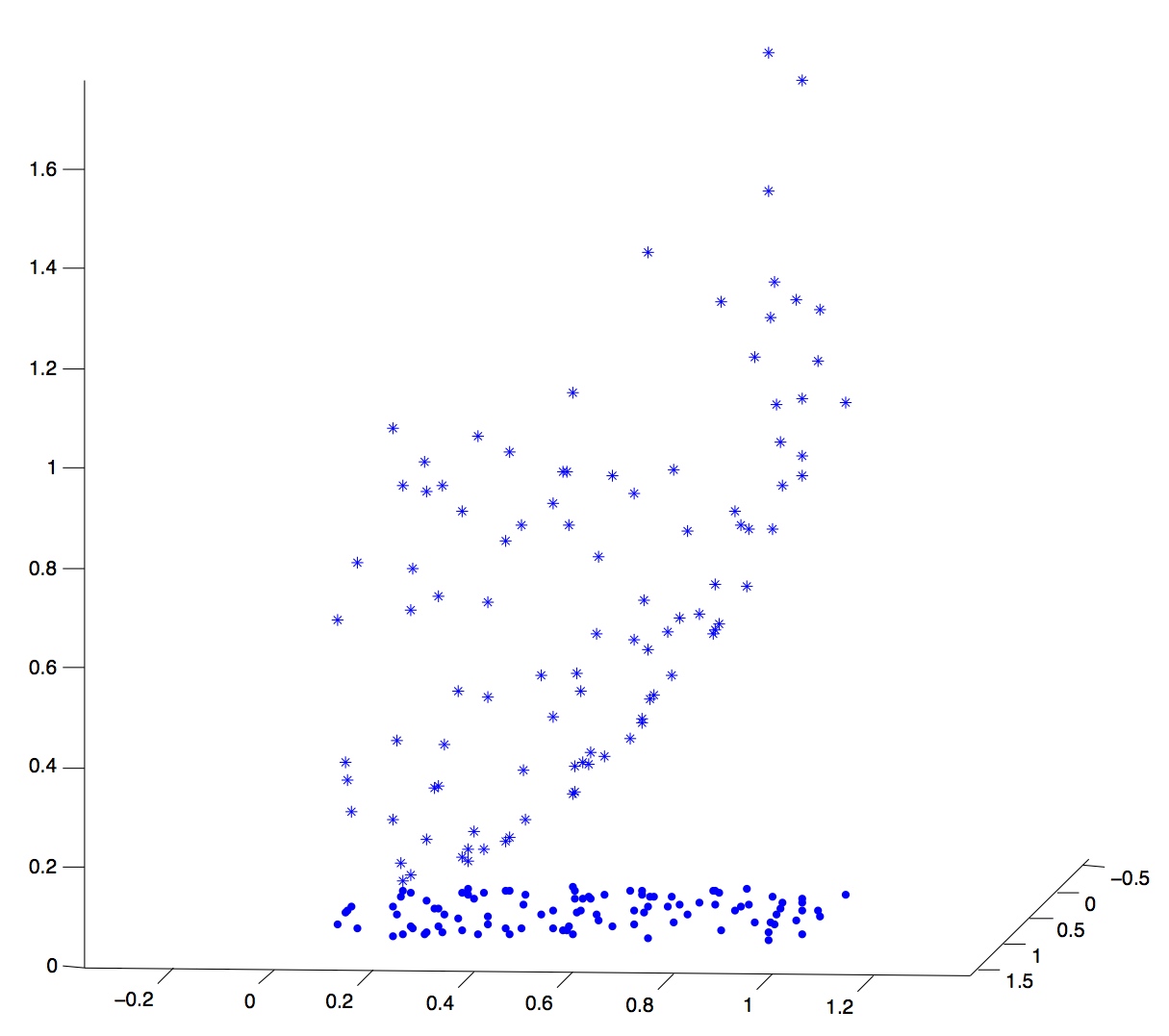}
}
\subfigure[The lower hull of the lifted points]{
\includegraphics[width=2.0in]{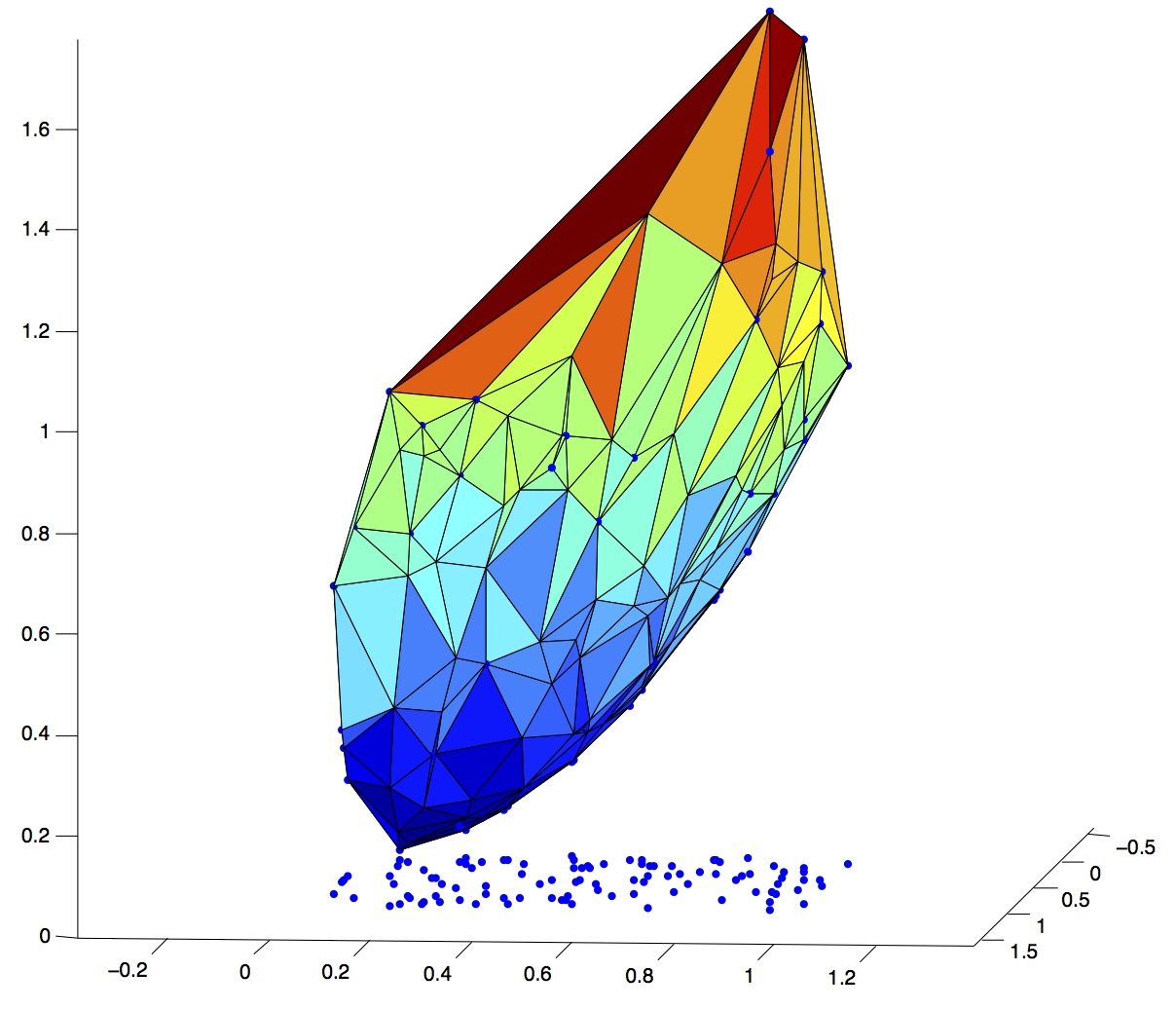}
}
\subfigure[View of the lower hull from below]{
\includegraphics[width=1.75in]{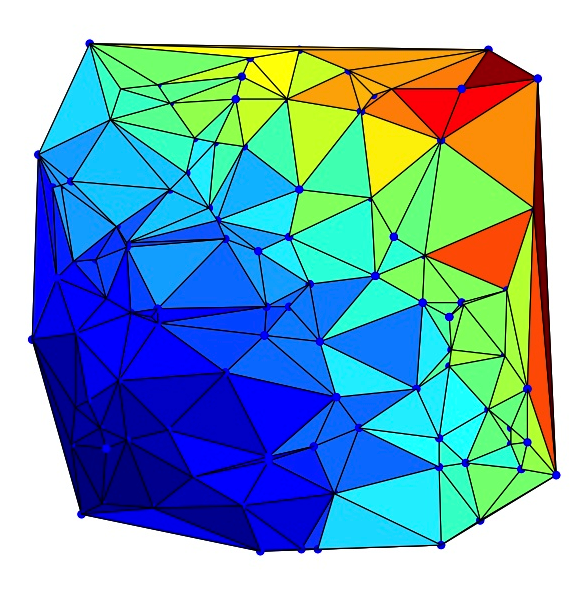}
}
\subfigure[Triangulation]{
\includegraphics[width=2.0in]{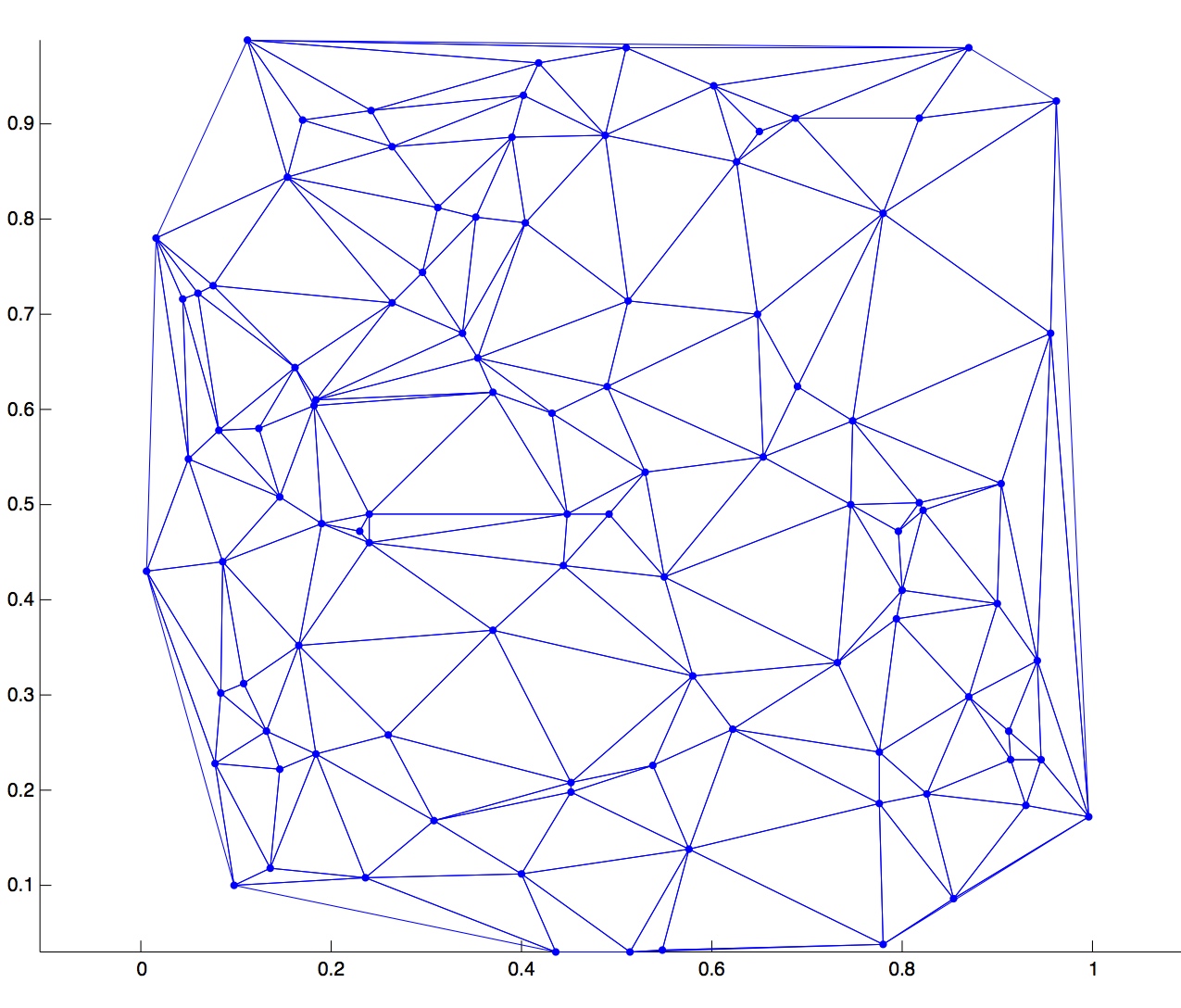}
}
\caption{Computing the dual triangulation of a power diagram
  via a convex hull in one dimension higher.}
\label{fig: lifted CH for DT}
\end{figure}

Thus to compute $\PD(V)$ and $\DT(V)$, we must compute a convex hull
in $\R^{d+1}$. This is a well studied problem with numerous algorithms
achieving optimal theoretical bounds in addition to others that work
efficiently in practice. We highlight some of these
algorithms here (let $C_{d+1}(N)$ denote the time needed to compute a
convex hull of $N$ points in $\R^{d+1}$):
\begin{enumerate}
\item
In \cite{clarkson:randSampGeoII1989, seidel:linProCHeasy1991,
  chazelle:optimalCH1993} worst case algorithms for general dimension
$d+1$ are given with complexity $C_{d+1}(N) = O(N\log N + N^{\lceil d/2
  \rceil})$. When $d=2$ this gives $O(N\log N)$ complexity. For higher
dimensions, the worst case is rather pessimistic when one considers
the average complexity over a family of convex hulls; see for example,
\cite{dwyer:avgConvHull1988}. 
\item
In \cite{seidel:chLogFace1986} an output sensitive algorithm for general dimension $d+1$ is given
which achieves $C_{d+1}(N) = O(N^2 + L\log L)$, where $L$ is the total
number of faces of the convex hull.
\item
For small dimensions greater than two, we have $C_4 = O((N+L)\log^2
L)$ \cite{chan:outputSensPolytope1995} and $C_5 = O((N+L) \log^3 L)$
\cite{amato:voronoiDivide1996}.
\item
The QuickHull algorithm \cite{barber:quickhull1996}, while not having provable bounds on its
complexity, is an output sensitive algorithm that empirically works
very well. It is able to handle numerical errors caused by
floating point arithmetic and is implemented for any dimension. In the
online code associated to this paper\footnote{Available at: \href{https://github.com/matthew-hirn/C-1-1-Interpolation}{\texttt{https://github.com/matthew-hirn/C-1-1-Interpolation}}}, we utilize this algorithm.
\end{enumerate}

\section{Overview of the algorithm} \label{sec: overview of algorithm}

The following is an overview of the algorithm for the \textsc{Jet
  interpolation problem} and the \textsc{Function interpolation
  problem}. The details of the algorithm are given in
Sections \ref{sec: compute gamma1}, \ref{sec: one time work}, and
\ref{sec: query work}. 

\subsection{Jet interpolation problem}

\noindent\underline{\textsc{Input (jet):}} The set $E \subset \R^d$
having $N$ points and the
$1$-field $P: E \rightarrow \PP$, which consists of the function
values $\{f_a\}_{a \in E} \subset \R$ and the gradients $\{D_af\}_{a
  \in E} \subset \R^d$. \\

\noindent\underline{\textsc{One time work, part I (jet):}} Compute $M
= \Lip(\nabla F)$, for which there are two options:
\begin{enumerate}

\item
$M = \Gamma^1(P;E)$ via direct calculation using \eqref{eqn: Gamma1
  alternate}, which requires $C \cdot d  \cdot N^2$ work and $C \cdot d
\cdot N$ storage.

\item[2.\footnotemark]\footnotetext{Not implemented in the online code
    available at:\\
    \href{https://github.com/matthew-hirn/C-1-1-Interpolation}{\texttt{https://github.com/matthew-hirn/C-1-1-Interpolation}}.}
Using the $\varepsilon$-WSPD of Section \ref{sec: wspd}, compute an
$M$ such that $cM \leq \Gamma^1(P;E) \leq M$, where $c < 1$ is an
absolute constant. The algorithm requires $(C \sqrt{d})^d \cdot N\log N$ work and
$(C \sqrt{d})^d \cdot N$ storage.

\end{enumerate}
\begin{remark}
The choice depends on the relative sizes of $N$ and $d$, as well as
the complexity of the remainder of the algorithm. For example, when
$d=2$ and $N$ is large, in practice (see Section \ref{sec: numerical
  simulations}) this step is the bottleneck of the entire algorithm if
one computes $M = \Gamma^1(P;E)$ exactly. In this case one might
want to utilize the second computation, since it gains a significant
speedup in $N$ while the exponential increase in $d$ is not much of a
factor since $d=2$. On the other hand, the second algorithm is
significantly more complicated to implement than the first, and scales
poorly for high dimensional interpolation problems. \\
\end{remark}

\noindent\underline{\textsc{Output, part I (jet):}} $M = \Lip(\nabla F)$. \\

\noindent\underline{\textsc{One time work, part II:}} Now we compute the
underlying geometrical structures of Wells' construction (Section
\ref{sec: wells}) using the set $E$, the $1$-field $P:E \rightarrow \PP$, and the value
$M$ computed from the \textsc{one time work, part I (jet)}. Recalling
Wells' set $\K$ from \eqref{eqn: wells special sets K}, define the
following two related sets:
\begin{align*}
&\widehat{\K} = \{ \widehat{S} \mid S \in \K \}, \\
&\K_{\ast} = \{ S_{\ast} \mid S \in \K\}.
\end{align*}
A key observation is that $\K_{\ast}$ is the power diagram of the
shifted points $\widetilde{E} = \{ \tilde{a} \mid a \in E\}$, and that $\widehat{\K}$
is its dual triangulation. The power function is,
\begin{equation*}
\pow (x,\tilde{a}) = \frac{4}{M} d_a(x),
\end{equation*}
where $d_a$ is the distance function defined in \eqref{eqn: wells da
  dist}, and the associated weight function is:
\begin{equation*}
w(\tilde{a}) = \frac{2}{M^2}|D_af|^2 - \frac{4}{M}f_a.
\end{equation*}
Sets $\{ a \}_{\ast} \in \K_{\ast}$ correspond to nonempty power cells
$\cell (\tilde{a})$, while sets $S_{\ast}$ for $\#(S) \geq 2$ are faces of the
power diagram $\PD (\widetilde{E})$. Additionally, there is a clear
bijective correspondence between $\K_{\ast}$ and $\widehat{\K}$, and
furthermore $S_{\ast} \subseteq S_{\ast}'$ if and only if $\widehat{S}'
\subseteq \widehat{S}$.

The main components of this part of the algorithm can be summarized as follows:
\begin{enumerate}

\item
Compute the shifted points $\widetilde{E} = \{ \tilde{a} \mid a \in E\}$. Requires
$O(N)$ work and $O(N)$ storage.

\item
Compute $\widehat{\K} = \DT (\widetilde{E})$ and $\K_{\ast} = \PD
(\widetilde{E})$ by first lifting the shifted points into $\R^{d+1}$
via the map $\lambda$ in \eqref{eqn: lambda lift}. Compute the convex
hull of the lifted points, and project back down into $\R^d$. This
gives $\DT (\widetilde{E})$. We store additionally for each $j$-dimensional face
$\widehat{S} \in \widehat{\K} = \DT (\widetilde{E})$ ($0 \leq j \leq d$), the
$(j-1)$-dimensional faces contained in $\widehat{S}$ (i.e., its
children) and the $(j+1)$-dimensional faces containing $\widehat{S}$
(i.e., its parents). Since $\K_{\ast} = \PD (\widetilde{E})$ is dual
to $\DT (\widetilde{E})$, we can derive it via duality.

The work of computing the convex hull is $O(N\log N +
N^{\lceil d/2 \rceil})$ and the storage for the convex hull
(equivalently the triangulation/power diagram) is $O(N^{\lceil d/2 \rceil})$
(see \cite{aurenhammer:powerDiagrams1987}). To calculate the children and
parents for each face in $\DT (\widetilde{E})$ requires $O (N^{d+1})$
work and $O(N^{\lceil d/2 \rceil + 1})$ storage. 

\item
Determine the point $S_C$ for each $S \in \K$ and compute
$d_S(S_C)$. Requires $O(N^{d+1})$ and $O(N^{\lceil d/2 \rceil})$ storage.

\item
Compute the sets $\{ T_S \}_{S \in \K}$ from $\widehat{\K} = \DT (\widetilde{E})$ and
$\K_{\ast} = \PD (\widetilde{E})$. Store each set $T_S$ as a pair
$(A_S,b_S)$, where $A_S$ is a matrix and $b_S$ is a vector, and $x \in
T_S$ if and only if $A_Sx \leq b_S$. Computing the full list of pairs
$\{ (A_S, b_S) \}_{S \in \K}$ requires $O(N^{d+2})$ work and
$O(N^{\lceil d/2 \rceil + 1})$ storage.\\

\end{enumerate}

\noindent\underline{\textsc{Query work:}} Given a query point $x \in \R^d$, one
must first determine which set $T_S$ it belongs to. There are two
methods to accomplish this task:
\begin{enumerate}
\item
A straightforward way is to check the inequalities $A_S x \leq b_S$
until one finds a pair $(A_S,b_S)$ that satisfies the condition for
$x$. In the worst case, one will have to check all of the
inequalities, which requires $O(N^{\lceil d/2 \rceil + 1})$ query work.

\item[{2.\footnotemark[\value{footnote}]}]
An alternate approach is to add an additional fifth step to the
\textsc{One time work, part II}. In this step, one places a tree
structure on the sets $\{T_S\}_{S \in \K}$ in which to each node we
associate a hyperplane and the leaves correspond to the sets
$\{T_S\}_{S \in \K}$. A query point $x$ is then passed down the tree
according to whether it lies to the left or right of the
hyperplane. If the tree is balanced, the query work is
$O(\log N)$. An algorithm that can guarantee a balanced tree can be found in
\cite{fuchs:polytopicPointLoc2010}; however, constructing the tree
requires solving an optimization problem for which it is not easy to
estimate the amount of work.
\end{enumerate}

Once the query point is placed in the correct set $T_S$, the function
$F_S$ is evaluated at $x$ and its gradient is computed. This requires
an amount of work that is dependent only on the dimension
$d$. \\

\noindent\underline{\textsc{Output, part II:}} $F(x)$ and $\nabla F(x)$ for
each query point $x \in \R^d$.

\subsection{Function interpolation problem}

In the case of the \textsc{Function interpolation problem}, we must
amend the inputs and the first part of the one time work; the
remainder of the algorithm is the same. \\

\noindent\underline{\textsc{Input (function):}} The set $E \subset \R^d$ and
the function $f: E \rightarrow \R$. \\

\noindent\underline{\textsc{One time work, part I (function):}} Compute $M =
\Lip(\nabla F)$ and $P_a = J_aF$, with $P_a(a) = f(a)$, for each $a
\in E$. The scalar $M$ satisfies $cM - 2\epsilon \leq \Gamma^1(f;E) \leq M$, with
$c < 1$ and $\epsilon > 0$. The algorithm uses the $\varepsilon$-WSPD in
conjunction with algorithms from convex optimization, and requires
$O(N^{7/2} (\log N)^{1/2} \log(N/\epsilon))$ work. \\

\noindent\underline{\textsc{Output, part I (function):}} $M =
\Lip(\nabla F)$ as well as a $1$-field $P$ such that $\Gamma^1(P; E) =
M$ and $P_a(a) = f(a)$ for all $a \in E$.

\section{Computing $\Gamma^1$} \label{sec: compute gamma1}

We now describe algorithms for computing $\Gamma^1$ or the
order of magnitude of $\Gamma^1$, for both the \textsc{Jet
  interpolation problem} and the \textsc{Function interpolation problem}.

\subsection{Jet interpolation problem}

For the \textsc{Jet interpolation problem}, as was discussed
in Section \ref{sec: overview of algorithm}, it is simple to compute
$\Gamma^1(P;E)$ exactly in $C \cdot d \cdot N^2$ work. In this section
we aim to improve the dependence on $N$ to be nearly linear, while
sacrificing a small amount of accuracy and some efficiency in the
dimension $d$. To that end, we prove the following theorem:

\begin{theorem} \label{thm: efficient L(PE)}
There is an algorithm, for which the inputs are the set $E \subset \R^d$ and
the 1-field $P: E \rightarrow \PP$, that computes a value $M$
satisfying:
\begin{equation*}
M/C_0 \leq \Gamma^1(P; E) \leq M,
\end{equation*}
where $C_0 > 1$ is an absolute constant. The algorithm
requires $C (d) \cdot N \log N$ work and $C(d) \cdot N$ storage.
\end{theorem}

The plan for proving Theorem \ref{thm: efficient L(PE)} is the
following. First we view the $\Gamma^1(P;E)$ functional from the
perspective of Whitney's Extension Theorem (i.e., Theorem \ref{thm: WET C11}) for
$C^{1,1}(\R^d)$. Once we formalize this concept, we can use the
$\varepsilon$-WSPD of Section \ref{sec: wspd}, since it is constructed to
handle interpolants in $C^m(\R^d)$ satisfying Whitney
conditions. The parameter $\varepsilon$ can be taken as any fixed number in
$(0, 1)$. The absolute constant $C_0$ of Theorem \ref{thm: efficient
  L(PE)} decreases linearly with $\varepsilon$ as $C_0 = 2 (1 +
\sqrt{2}) (3 + 23\varepsilon)$. However the dimensional
constant $C(d)$ increases exponentially with $d$, and in particular
$C(d) = (C \cdot \sqrt{d} / \varepsilon )^d$ for any $\varepsilon \leq
1/2$  (see Remark \ref{rem: C(e,d)}).

Concerning the first part of our approach, recall that if we take the infimum over all
$M$ satisfying the Whitney conditions \ref{W0} and \ref{W1}, then we
obtain a value that is within a constant $C'(d)$ of
$\|P\|_{C^{1,1}(E)}$. The main contribution of
\cite{legruyer:minLipschitzExt2009} is to refine \ref{W0} and \ref{W1}
so that $C'(d) = 1$; this is $\Gamma^1(P;E)$. Indeed, referring to
the alternate form of $\Gamma^1(P;E)$ given in \eqref{eqn: Gamma1
  alternate}, the functional $A$
corresponds to \ref{W0}, the functional $B$ corresponds to
\ref{W1}, and $\Gamma^1(P;E)$ pieces them together. Note there are some
small, but significant differences. In particular, the functional $A$
is essentially a symmetric version of \ref{W0}; using one is
equivalent to using the other, up to a factor of two. The functional
$B$ though, merges all of the partial derivative information into one
condition, unlike \ref{W1}. Thus they are equivalent only up to a
factor of $d$, the dimension of the Euclidean space we are working
in. For the algorithm in this section, we will use the functional $B$
since it is both simpler and more useful than \ref{W1}, but use \ref{W0}
instead of $A$. Additionally, we will treat them separately instead of together like in
$\Gamma^1(P;E)$; Lemma \ref{lem: first gamma estimate} contains the details.

For the 1-field $P: E \rightarrow \PP$,
define the functional $\widetilde{A}$, which is essentially the same as \ref{W0}:
\begin{equation*}
\widetilde{A}(P;a,b) = \frac{|P_a(a) - P_b(a)|}{|a-b|^2},
\quad a,b \in E.
\end{equation*}
Additionally, set
\begin{equation*}
\widetilde{\Gamma}^1(P;E) = \max_{\substack{a,b \in E \\ a
    \neq b}} \Big\{ \max\{\widetilde{A}(P;a,b), B(P;a,b)\} \Big\}.
\end{equation*}
The functional $\widetilde{\Gamma}^1(P;E)$ is more easily approximated
via the $\varepsilon$-WSPD than $\Gamma^1(P;E)$. Furthermore, as the
following lemma shows, they have the same order of magnitude.

\begin{lemma} \label{lem: first gamma estimate}
For any finite set $E \subset \R^d$ and any 1-field $P: E \rightarrow \PP$,
\begin{equation*}
\widetilde{\Gamma}^1(P;E) \leq \Gamma^1(P;E) \leq 2(1+ \sqrt{2})
\widetilde{\Gamma}^1(P;E).
\end{equation*}
\end{lemma}

\begin{proof}
To bridge the gap between $\Gamma^1(P;E)$ and
$\widetilde{\Gamma}^1(P;E)$, we first consider
\begin{equation*}
\overline{\Gamma}^1(P;E) = \max_{\substack{a,b \in E \\ a \neq b}}
\Big\{ \max \{A(P;a,b), B(P;a,b)\} \Big\}.
\end{equation*}
Clearly $\overline{\Gamma}^1(P;E) \leq \Gamma^1(P;E)$. Furthermore,
\begin{align*}
\Gamma^1(P;E) &= \max_{\substack{a,b \in E \\ a \neq b}} \sqrt{A(P;a,b)^2
  + B(P;a,b)^2} + A(P;a,b) \\
&\leq \sqrt{\overline{\Gamma}^1(P;E)^2 + \overline{\Gamma}^1(P;E)^2} + \overline{\Gamma}^1(P;E) \\
&\leq (1 + \sqrt{2})\overline{\Gamma}^1(P;E).
\end{align*}
Thus $\Gamma^1(P;E)$ and $\overline{\Gamma}^1(P;E)$ have the same order of
magnitude, and in particular,
\begin{equation}\label{eqn: Gamma and Gamma'}
\overline{\Gamma}^1(P;E) \leq \Gamma^1(P;E) \leq (1+\sqrt{2})\overline{\Gamma}^1(P;E).
\end{equation}

Now let us consider $\overline{\Gamma}^1(P;E)$ and
$\widetilde{\Gamma}^1(P;E)$ (which means considering
$A(P;a,b)$ and $\widetilde{A}(P;a,b)$). First,
\begin{align*}
|P_a(a) - P_b(a) + P_a(b) - P_b(b)| &\leq |P_a(a) - P_b(a)| + |P_a(b)
- P_b(b)| \\
&\leq 2\widetilde{\Gamma}^1(P;E) |a-b|^2,
\end{align*}
and so, $\overline{\Gamma}^1(P;E) \leq 2\widetilde{\Gamma}^1(P;E)$. For a
reverse inequality, we note,
\begin{equation*}
|P_a(a) - P_b(a) + P_a(b) - P_b(b)| = |2(P_a(a) - P_b(a)) + (\nabla
P_b - \nabla P_a) \cdot (a-b)|.
\end{equation*}
Thus,
\begin{align*}
2|P_a(a) - P_b(a)| &\leq \overline{\Gamma}^1(P;E)|a-b|^2 + |(\nabla P_b - \nabla
P_a) \cdot (a-b)| \\
&\leq 2\overline{\Gamma}^1(P;E) |a-b|^2,
\end{align*}
which yields $\widetilde{\Gamma}^1(P;E) \leq
\overline{\Gamma}^1(P;E)$. Combining the two inequalities,
\begin{equation}\label{eqn: Gamma' and Gammatilde}
\widetilde{\Gamma}^1(P;E) \leq \overline{\Gamma}^1(P;E) \leq 2\widetilde{\Gamma}^1(P;E).
\end{equation}
Putting \eqref{eqn: Gamma and Gamma'} and \eqref{eqn: Gamma' and
  Gammatilde} together completes the proof.
\end{proof}

We will also need the following simple lemmas.
\begin{lemma}\label{lem: wspd estimates}
Let $(\T,\W)$ be a $\varepsilon$-WSPD, $(\Lambda_1, \Lambda_2) \in
\W$, $x,x',x'' \in \cup \Lambda_1$, and $y,y' \in \cup
\Lambda_2$. Then,
\begin{align*}
&|x'-x''| \leq \varepsilon |x-y| \\
&|x'-y'| \leq (1+2\varepsilon) |x-y|.
\end{align*}
\end{lemma}
\begin{proof}
Use the definition of $\varepsilon$-separated.
\end{proof}

\begin{lemma}\label{lem: polynomial shift}
Suppose $p \in \PP$, $x \in \R^d$, $\delta > 0$, and $M > 0$
satisfy
\begin{align*}
&|p(x)| \leq M \delta^2 \\
&|\nabla p(x)| \leq M \delta.
\end{align*}
Then, for any $y \in \R^d$,
\begin{equation*}
|p(y)| \leq M(\delta + |x-y|)^2.
\end{equation*}
\end{lemma}
\begin{proof}
Using Taylor's Theorem,
\begin{align*}
|p(y)| &= |p(x) + \nabla p(x) \cdot (y-x)| \\
&\leq |p(x)| + |\nabla p(x)| |x-y| \\
&\leq M\delta^2 + M\delta |x-y| \\
&\leq M(\delta + |x-y|)^2.
\end{align*}
\end{proof}

\begin{proof}[Proof of Theorem \ref{thm: efficient L(PE)}]
Set:
\begin{align*}
\widetilde{A}(P;E) &= \max_{\substack{a,b \in E \\ a \neq
    b}} \widetilde{A}(P;a,b), \\
B(P;E) &= \max_{\substack{a,b \in E \\ a \neq b}} B(P;a,b),
\end{align*}
and note that $\widetilde{\Gamma}^1 (P; E) = \max \{ \widetilde{A} (P;
E), \, B (P; E) \}$. 

Our algorithm works as follows. For now, let $0 < \varepsilon < 1$ be
arbitrary and invoke the algorithm from Theorem \ref{thm: F-K
  WSPD}. This gives us an $\varepsilon$-WSPD $(\T,\W)$ in
$C(\varepsilon, d) N\log N$ work and using
$C (\varepsilon, d) N$ storage. For each $(\Lambda_1,
\Lambda_2) \in \W$, pick a representative pair $(a_{\Lambda_1},
a_{\Lambda_2}) \in \cup \Lambda_1 \times \cup
\Lambda_2$. Additionally, for each $S \in \T$, pick a
representative $a_S \in S$.

We now compute the following:
\begin{align}
\widetilde{A}_1(P; \T, \W) &=  \max_{(\Lambda_1, \Lambda_2) \in \W}
  \widetilde{A} (P; a_{\Lambda_1}, a_{\Lambda_2}), \nonumber \\
\widetilde{A}_2(P; \T, \W) &= \max_{(\Lambda_1, \Lambda_2) \in \W}
\, \max_{i=1,2} \, \max_{S \in \Lambda_i} \widetilde{A}(P;
a_{\Lambda_i}, a_S), \label{eqn: Atilde wspd}\\
\widetilde{A}_3(P; \T, \W) &= \max_{S \in \T} \max_{a \in S}
\widetilde{A}(P; a, a_S), \nonumber
\end{align}
and
\begin{align*}
B_1(P; \T, \W) &= \max_{(\Lambda_1, \Lambda_2) \in \W} B(P;
a_{\Lambda_1}, a_{\Lambda_2}), \\
B_2(P; \T, \W) &= \max_{(\Lambda_1, \Lambda_2) \in \W}
\, \max_{i=1,2} \, \max_{S \in \Lambda_i} B(P;
a_{\Lambda_i}, a_S), \\
B_3(P; \T, \W) &= \max_{S \in \T} \max_{a \in S}
B(P; a, a_S).
\end{align*}
Additionally, compute:
\begin{align*}
\widetilde{A} (P; \T, \W) &= \max_{i=1, 2, 3} \widetilde{A}_i (P; \T,
\W), \\
B (P; \T, \W) &= \max_{i=1,2,3} B_i (P; \T, \W),
\end{align*}
as well as
\begin{equation} \label{eqn: approx gamma}
\widetilde{\Gamma}^1 (P; \T, \W) = \max \{ \widetilde{A} (P; \T, \W), \,
B (P; \T, \W) \}.
\end{equation}
Using properties \ref{item: F-K 6} and \ref{item: F-K 7} from Section
\ref{sec: wspd}, we see that computing
$\widetilde{\Gamma}^1(P;\T,\W)$ requires
$C(\varepsilon, d) N\log N$ work and
$C(\varepsilon, d) N$ storage.

Now we show that $\widetilde{\Gamma}^1(P;\T,\W)$ has the same order
of magnitude as $\widetilde{\Gamma}^1(P;E)$. Clearly,
$\widetilde{\Gamma}^1(P; \T, \W) \leq
\widetilde{\Gamma}^1(P;E)$. For the other inequality, we break
$\widetilde{\Gamma}^1 (P; E)$ into its two parts. Since
$\widetilde{\Gamma}^1(P;E) = \max\{\widetilde{A}(P;E), \, B(P;E)\}$,
it suffices to show that $\widetilde{A} (P; E)$ and $B (P; E)$ are
each bounded from above by $C(\varepsilon) \cdot \widetilde{\Gamma}^1 (P; \T,
\W)$. 

Set:
\begin{equation*}
\widetilde{\Gamma}_i^1 (P; \T, \W) = \max
\{ \widetilde{A}_i (P; \T, \W), B_i (P; \T, \W) \}, \quad i = 1, 2, 3,
\end{equation*}
and note $\widetilde{\Gamma}_i^1 (P; \T, \W) \leq \widetilde{\Gamma}^1
(P; \T, \W)$ for each $i = 1, 2, 3$. 

Let $a,b \in E$, $a \neq b$. By
properties \ref{item: F-K 1} and \ref{item: F-K 2} of Section
\ref{sec: wspd}, there is a unique pair $(\Lambda_1, \Lambda_2) \in
\W$ such that $(a,b) \in \cup \Lambda_1 \times \cup
\Lambda_2$. Additionally, by the definition of $(\T,\W)$, there exists
a set $S \in \Lambda_1$ such that $a \in S$ and a set $T \in \Lambda_2$ such
that $b \in T$.

We first bound $\widetilde{A}(P; E)$. Using the
triangle inequality, the definition of $\widetilde{A}_3(P; \T, \W)$, and
Lemma \ref{lem: wspd estimates},
\begin{align}
|P_a(a) - P_b(a)| &\leq |P_a(a) - P_{a_S}(a)| + |P_{a_S}(a) - P_b(a)|
\nonumber \\
&\leq \widetilde{A}_3(P; \T, \W) \cdot |a-a_S|^2 + |P_{a_S}(a) - P_b(a)| \nonumber \\
&\leq \varepsilon \cdot \widetilde{\Gamma}^1 (P; \T, \W) \cdot |a-b|^2 + |P_{a_S}(a) -
P_b(a)|. \label{eqn: A estimate 1}
\end{align}
We continue with the second term of the right hand side of \eqref{eqn:
  A estimate 1}. Using the triangle inequality, Lemma \ref{lem: polynomial shift}
applied to $P_{a_S} - P_{a_{\Lambda_1}}$ with $x = a_{\Lambda_1}$, $y =
a$ and $M = \widetilde{\Gamma}_2^1 (P; \T, \W)$, as well as Lemma
\ref{lem: wspd estimates}, we obtain:
\begin{align}
|P_{a_S}(a) - P_b(a)| &\leq |P_{a_S}(a) - P_{a_{\Lambda_1}}(a)| +
|P_{a_{\Lambda_1}}(a) - P_b(a)| \nonumber \\
&\leq \widetilde{\Gamma}^1_2 (P; \T, \W) \cdot (|a_S - a_{\Lambda_1}| + |a -
a_{\Lambda_1}|)^2 + |P_{a_{\Lambda_1}}(a) - P_b(a)| \nonumber \\
&\leq 4\varepsilon^2 \cdot \widetilde{\Gamma}^1 (P; \T, \W) \cdot
  |a-b|^2 + |P_{a_{\Lambda_1}}(a) - P_b(a)|. \label{eqn: A estimate 2}
\end{align}
To bound the second term of the right hand side of \eqref{eqn:
  A estimate 2}, we use the triangle inequality, Lemma \ref{lem:
  polynomial shift} applied to $P_b - P_{a_T}$ with $x = b$, $y = a$
and $M = \widetilde{\Gamma}_3^1 (P; \T, \W)$, and Lemma \ref{lem: wspd estimates}:
\begin{align}
|P_{a_{\Lambda_1}}(a) - P_b(a)| &\leq |P_b(a) - P_{a_T}(a)| +
|P_{a_T}(a) - P_{a_{\Lambda_1}}(a)| \nonumber \\
&\leq \widetilde{\Gamma}^1_3 (P; \T, \W) \cdot (|b-a_T| + |a-b|)^2 + |P_{a_T}(a) -
P_{a_{\Lambda_1}}(a)| \nonumber \\
&\leq (1+\varepsilon)^2 \cdot \widetilde{\Gamma}^1 (P; \T, \W) \cdot
  |a-b|^2 + |P_{a_T}(a) - P_{a_{\Lambda_1}}(a)|. \label{eqn: A estimate 3}
\end{align}
Finally, for the second term of the right hand side of \eqref{eqn:
  A estimate 3}, we use the triangle inequality, Lemma \ref{lem:
  polynomial shift} applied to $P_{a_T} - P_{a_{\Lambda_2}}$ with $x =
a_{\Lambda_2}$, $y = a$ and $M = \widetilde{\Gamma}_2^1(P; \T, \W)$,
Lemma \ref{lem: polynomial shift} applied to $P_{a_{\Lambda_2}} -
P_{a_{\Lambda_1}}$ with $x = a_{\Lambda_1}$, $y = a$ and $M =
\widetilde{\Gamma}_1^1(P; \T, \W)$, as well as Lemma \ref{lem: wspd estimates}:
\begin{align}
|P_{a_T}(a) &- P_{a_{\Lambda_1}(a)}(a)| \nonumber \\
&\leq |P_{a_T}(a) -
P_{a_{\Lambda_2}}(a)| + |P_{a_{\Lambda_2}}(a) - P_{a_{\Lambda_1}}(a)|
\nonumber \\
&\leq \widetilde{\Gamma}_2^1 (P; \T, \W) \cdot ( |a_T - a_{\Lambda_2}|
  + |a - a_{\Lambda_2}|)^2 +  |P_{a_{\Lambda_2}}(a) -
  P_{a_{\Lambda_1}}(a)| \nonumber \\
& \leq (1+3\varepsilon)^2 \cdot \widetilde{\Gamma}^1 (P; \T, \W)
  \cdot |a-b|^2 + \widetilde{\Gamma}^1_1 (P; \T, \W) \cdot (|a_{\Lambda_2}-a_{\Lambda_1}| +
|a-a_{\Lambda_1}|)^2 \nonumber \\
&\leq 2 (1 + 3\varepsilon)^2 \cdot \widetilde{\Gamma}^1 (P; \T, \W)
  \cdot |a-b|^2. \label{eqn: A estimate 4}
\end{align}
Putting \eqref{eqn: A estimate 1}, \eqref{eqn: A estimate 2},
\eqref{eqn: A estimate 3}, \eqref{eqn: A estimate 4} together, we get:
\begin{equation}\label{eqn: A estimate 5}
|P_a(a) - P_b(a)| \leq (3 + 23\varepsilon) \cdot \widetilde{\Gamma}^1
(P; \T, \W) \cdot |a-b|^2.
\end{equation}

The proof for $B (P; E)$ proceeds along similar lines. Using the same
sequence of triangle inequalities, in addition to several applications
of Lemma \ref{lem: wspd estimates}, yields:
\begin{align}
|D_af - D_bf| \leq{}& |D_af - D_{a_S}f| + |D_{a_S}f - D_{a_{\Lambda_1}}f|
  + |D_{a_{\Lambda_1}}f - D_{a_{\Lambda_2}}f| \nonumber \\
&+ |D_{a_{\Lambda_2}}f - D_{a_T}f| + |D_{a_T}f - D_bf| \nonumber \\
\leq{}& B_3 (P; \T, \W) \cdot |a - a_S| + B_2 (P; \T, \W) \cdot |a_S -
        a_{\Lambda_1}| \nonumber \\
&+ B_1 (P; \T, \W) \cdot |a_{\Lambda_1} - a_{\Lambda_2}| + B_2 (P; \T,
  \W) \cdot |a_{\Lambda_2} - a_T| \nonumber \\
&+ B_1 (P; \T, \W) \cdot |a_T - b| \nonumber \\
\leq{}& (1 + 6 \varepsilon) \cdot \widetilde{\Gamma}^1 (P; \T, \W)
        \cdot |a-b|. \label{eqn: B estimate}
\end{align}

Combining \eqref{eqn: A estimate 5} and \eqref{eqn: B estimate} gives:
\begin{equation*}
\widetilde{\Gamma}^1 (P; \T, \W) \leq \widetilde{\Gamma}^1 (P; E) \leq
(3 + 23\varepsilon) \cdot \widetilde{\Gamma}^1 (P; \T, \W).
\end{equation*}
Now apply Lemma \ref{lem: first gamma estimate} to obtain:
\begin{equation*}
\widetilde{\Gamma}^1 (P; \T, \W) \leq \Gamma^1 (P; E) \leq 2 (1 +
\sqrt{2}) (3 + 23\varepsilon) \cdot \widetilde{\Gamma}^1 (P; \T, \W).
\end{equation*}
The proof is completed by selecting any $\varepsilon \in (0, 1)$.
\end{proof}

\subsection{Function interpolation problem} \label{sec: function
  interpolation problem}

\subsubsection{Convex optimization}

For the \textsc{Function interpolation problem}, we have only the
function values $f: E \rightarrow \R$ and we must compute:
\begin{equation*}
\Gamma^1(f;E) = \inf \{ \Gamma^1(P;E) \mid P_a(a) = f(a) \text{ for
  all } a \in E\}.
\end{equation*}
Recall that $P_a(x) = f_a + D_af \cdot (x-a)$. Since we must have
$P_a(a) = f_a = f(a)$, the values $\{f_a\}_{a \in E}$ are
fixed. Additionally, the set $E$ is fixed. Therefore, we must solve
for the gradients $\{D_af\}_{a \in E}$ that minimize
$\Gamma^1(P;E)$. Thus in this section we view $\Gamma^1(P;E)$ as
a function of the gradients. In order to clarify this point, let $E =
\{a_k\}_{k=1}^N$ be an indexation of $E$ and define a new variable $Y
= (y_1,\ldots,y_N) \in \R^{dN}$ with $y_k \in \R^d$ for each
$k=1,\ldots,N$. Let $g: \R^{dN} \rightarrow \R$ be defined as:
\begin{equation*}
g(Y) = \Gamma^1(P;E), \quad P_{a_k}(x) = f(a_k) + y_k \cdot (x-a_k).
\end{equation*}

The function $g: \R^{dN} \rightarrow \R$ is a convex
function and since $\Gamma^1(P;E) > 0$ it is piecewise twice
differentiable. Therefore we can use algorithms from convex
optimization to solve for $\Gamma^1(f;E)$. Indeed, consider the
following unconstrained convex optimization problem:
\begin{equation} \label{eqn: uncon conv opt gamma}
\text{minimize} \quad g(Y).
\end{equation}
The value of \eqref{eqn: uncon conv opt gamma} is
$\Gamma^1(f;E)$. Additionally, if $Y^{\star} =
(y_1^{\star},\ldots,y_N^{\star})$ is the minimizer, then the $1$-field
$P^{\star}$ defined by
\begin{equation*}
P_{a_k}^{\star}(x) = f(a_k) + y_k^{\star} \cdot (x-a_k),
\end{equation*}
achieves the value $\Gamma^1(f;E)$. 
One can solve \eqref{eqn: uncon conv opt gamma} using Newton's
method. A rigorous estimate for the number of iterations (in
particular the number of Newton steps) is difficult to compute due to
the square root in the $\Gamma^1$ functional, but we can examine a
related convex optimization problem to get a related estimate.

Recall the functional $\widetilde{\Gamma}^1(P;E)$, which has
the same order of magnitude as $\Gamma^1$, and was defined as:
\begin{equation*}
\widetilde{\Gamma}^1(P;E) = \max_{\substack{a,b, \in E \\ a \neq b}}
\left\{ \widetilde{A}(P;a,b), B(P;a,b) \right\}.
\end{equation*}
Define $\widetilde{\Gamma}^1(f;E)$ analogously to $\Gamma^1(f;E)$, 
\begin{equation*}
\widetilde{\Gamma}^1(f;E) = \inf \{ \widetilde{\Gamma}^1(P;E) \mid
P_a(a) = f(a) \text{ for all } a \in E \},
\end{equation*}
and similarly define $\tilde{g}: \R^{dN} \rightarrow \R$ analogously
to $g$. Then the following unconstrained convex optimization problem solves for
$\widetilde{\Gamma}^1(f;E)$: 
\begin{equation} \label{eqn: uco for gamma tilde}
\text{minimize} \quad \tilde{g}(Y).
\end{equation}
We can rewrite \eqref{eqn: uco for gamma tilde} in a form that is more
easily accessible and that utilizes only continuous, twice
differentiable functions (as opposed to $\tilde{g}$ which is piecewise
such). Related to the functional $\widetilde{A}$, define two families
of functions $\alpha_{j,k}^+: \R^{dN+1} \rightarrow \R$ and
$\alpha_{j,k}^-: \R^{dN+1} \rightarrow \R$,
\begin{align*}
\alpha_{j,k}^+(Y,M) &= y_k \cdot (a_k-a_j) - M|a_j-a_k|^2 + f(a_j) -
f(a_k), \quad j,k = 1,\ldots,N, \\
\alpha_{j,k}^-(Y,M) &= y_k \cdot (a_j-a_k) - M|a_j-a_k|^2 + f(a_k) -
f(a_j), \quad j,k = 1,\ldots,N,
\end{align*}
where $M \in \R$. Additionally, for the functional $B$ define $\beta_{j,k}:
\R^{dN+1} \rightarrow \R$,
\begin{equation*}
\beta_{j,k}(Y,M) = |y_j|^2 + |y_k|^2 - 2y_j \cdot y_k -
M^2|a_j-a_k|^2, \quad j,k = 1,\ldots,N.
\end{equation*}
Then the following optimization problem is equivalent to \eqref{eqn:
  uco for gamma tilde}:
\begin{align}
\text{minimize}& \quad M, \label{eqn: QP for gamma tilde} \\
\text{subject to}& \quad \alpha_{j,k}^+(Y,M) \leq 0, \quad \forall
\, j,k = 1,\ldots,N, \nonumber \\
& \quad \alpha_{j,k}^-(Y,M) \leq 0, \quad \forall \, j,k =
1,\ldots,N, \nonumber \\
& \quad \beta_{j,k}(Y,M) \leq 0, \quad \forall \, j,k =
1,\ldots,N. \nonumber
\end{align}
Indeed, if $(Y^{\star},M^{\star})$ is the minimizer, then
$\widetilde{\Gamma}^1(f;E) = M^{\star}$ and $Y^{\star}$ defines the
gradients of the $1$-field $P^{\star}$ such that
$\widetilde{\Gamma}^1(P^{\star}; E) = M^{\star}$. 

Constrained convex optimization problems can be solved using interior
point methods. In Appendix \ref{sec: convex optimization} we describe
a particular form of the barrier method that iteratively solves a
sequence of unconstrained optimization problems with Newton's
method. Each iteration is referred to as a Newton step.

The $\alpha$ functions are linear and the $\beta$ functions are
quadratic; therefore, \eqref{eqn: QP for gamma tilde} is a
quadratically constrained quadratic program (QCQP). Thus Theorem
\ref{thm: newton step bound} from Appendix \ref{sec: convex
  optimization} applies and we see that the number of Newton steps
required for the barrier method to solve \eqref{eqn: QP for gamma
  tilde} to within (additive) accuracy $\epsilon$ is $O(N\log (N/\epsilon))$. 

The cost of each Newton step
can be derived from equations \eqref{eqn: cost per Newton step} and
\eqref{eqn: Newton step H cost for barrier} (also in Appendix
\ref{sec: convex optimization}). Without considering
any structure in the problem, the amount of work is $O(N^4)$ due to
the cost of forming the relevant Hessian matrix $H$. However,
each $\alpha$ and $\beta$ function depends only on $2d+1$ variables;
therefore the cost of forming $H$ is in fact $O(N^2)$. In this
case the Cholesky factorization for computing $H^{-1}$ will dominate
with $O(N^3)$ work per Newton step. Thus the total work is $O(N^4 \log
(N/\epsilon))$. We collect this result in the following proposition.

\begin{proposition}
There is an algorithm, for which the inputs are the set $E \subset
\R^d$, the function $f: E \rightarrow \R$ and a parameter $\epsilon >
0$, that computes a value $\widetilde{M}$ satisfying:
\begin{equation*}
\widetilde{M} - \epsilon \leq \widetilde{\Gamma}^1(f; E) \leq \widetilde{M} + \epsilon,
\end{equation*}
as well as a $1$-field $P$ such that $\Gamma^1(P; E) = \widetilde{M}$ and $P_a(a)
= f(a)$ for all $a \in E$. The algorithm requires $O (N^4 \log (N / \epsilon))$ work.
\end{proposition}

As a corollary we have:

\begin{corollary}
There is an algorithm, for which the inputs are the set $E \subset
\R^d$, the function $f: E \rightarrow \R$ and a parameter $\epsilon >
0$, that computes a value $M$ satisfying:
\begin{equation*}
M/C - 2\epsilon \leq \Gamma^1(f; E) \leq M,
\end{equation*}
where $C = 2(1+\sqrt{2})$. The algorithm also outputs a $1$-field $P$
such that $\Gamma^1(P; E) = M$ and $P_a(a) = f(a)$ for all $a \in
E$. The work required is $O (N^4 \log (N / \epsilon))$.
\end{corollary}

\subsubsection{Convex optimization + $\varepsilon$-WSPD} \label{sec:
  convex opt + wspd}

As in the \textsc{Jet interpolation problem}, one can reduce the
amount of work by utilizing the $\varepsilon$-WSPD and computing a
value $\widetilde{M}$ that is within a multiplicative constant of
$\widetilde{\Gamma}^1(f;E)$. 

Recall from the proof of Theorem \ref{thm: efficient L(PE)} the
constant $\widetilde{\Gamma}^1(P; \T, \W)$ defined in \eqref{eqn:
  approx gamma}, which has the same order of magnitude as $\Gamma^1(P;E)$. It 
is computed over pairs of points $(E \times E)_{\T, \W} \subset E \times E$ derived from the
$\varepsilon$-WSPD $(\T, \W)$ of Section \ref{sec: wspd}; recall these
pairs are (see \eqref{eqn: Atilde wspd}):
\begin{align*}
(E \times E)_{\T, \W} ={} &\{ (a_{\Lambda_1}, a_{\Lambda_2}) \mid
(\Lambda_1, \Lambda_2) \in \W \} \, \cup \\
&\{ (a_{\Lambda_i}, a_S) \mid (\Lambda_1, \Lambda_2) \in \W, \, i =
  1,2, \, S \in \Lambda_i \} \, \cup \\
&\{ (a, a_S) \mid S \in \T, \, a \in S \}.
\end{align*}
Let $(\Z_N \times \Z_N)_{\T,\W} \subset \Z_N \times \Z_N$ denote the indices of
these pairs. Consider the following optimization problem:

\begin{align}
\text{minimize}& \quad M, \label{eqn: QP for gamma tilde wspd} \\
\text{subject to}& \quad \alpha_{j,k}^+(Y,M) \leq 0, \quad \forall
\, (j,k) \in (\Z_N \times \Z_N)_{\T,\W}, \nonumber \\
& \quad \alpha_{j,k}^-(Y,M) \leq 0, \quad \forall \, (j,k) \in
(\Z_N \times \Z_N)_{\T,\W}, \nonumber \\
& \quad \beta_{j,k}(Y,M) \leq 0, \quad \forall \, (j,k) \in (\Z_N
\times \Z_N)_{\T,\W}. \nonumber
\end{align}
The minimizer of \eqref{eqn: QP for gamma tilde wspd} is
$\widetilde{\Gamma}^1(f; \T, \W)$, which is defined as:
\begin{equation*}
\widetilde{\Gamma}^1(f; \T, \W) = \inf \{ \widetilde{\Gamma}^1(P; \T,
\W) \mid P_a(a) = f(a) \text{ for all } a \in E \}.
\end{equation*}
By Theorem \ref{thm: efficient L(PE)} it has the same order of magnitude
as $\Gamma^1(f;E)$. Additionally, by construction of the
$\varepsilon$-WSPD $(\T,\W)$, $(\Z_N \times \Z_N)_{\T,\W}$ has only
$O(N\log N)$ pairs of points. Thus Theorem \ref{thm: newton step
  bound} implies the number of Newton steps required for the barrier
method is $O((N\log N)^{1/2}\log(N/\epsilon))$.

The cost per Newton step is harder to bound rigorously. The amount of work to
form the relevant Hessian matrix $H$ is $O(N \log N)$. Furthermore, the Hessian
matrix is sparse, with only $O(N \log N)$ nonzero entries. Thus, for the Cholesky
factorization, we can use a sparse factorization
algorithm. However, an exact bound on the work required depends
on the sparsity pattern. In fact we know the pattern, since it can be derived
from the $\varepsilon$-WSPD, however, to the best of our knowledge
there is no theorem relating well separated pair decompositions and
sparse Cholesky factorization algorithms. For sparse matrices
corresponding to planar graphs, the storage is $O(N\log N)$ and the work is
$O(N^{3/2})$ for the Cholesky factorization
\cite{lipton:generalNestedDissection1979}; one might hope a similar
theorem could be proved for the Hessian matrix derived from the
$\varepsilon$-WSPD. As it stands currently, we are guaranteed no more
than $O(N^{7/2} (\log N)^{1/2} \log (N/\epsilon))$ work:

\begin{proposition}
There is an algorithm, for which the inputs are the set $E \subset
\R^d$, the function $f: E \rightarrow \R$ and a parameter $\epsilon >
0$, that computes a value $M$ satisfying:
\begin{equation*}
M/C_0 - 2\epsilon \leq \Gamma^1(f; E) \leq M,
\end{equation*}
where $C_0$ is the same absolute constant of Theorem \ref{thm:
  efficient L(PE)}. The algorithm also outputs $1$-field $P$ such
that $\Gamma^1(P; E) = M$ and $P_a(a) = f(a)$ for all $a \in E$. The
work required is $O(N^{7/2} (\log N)^{1/2} \log (N/\epsilon))$.
\end{proposition}

\section{One time work, part II} \label{sec: one time work}

At this point in the algorithm we have computed a value $M =
\Lip(\nabla F)$ such that $M = \Gamma^1$ (either the jet or function
version) or $M$ has the same order of magnitude as
$\Gamma^1$. Additionally  we have a 1-field $P: E \rightarrow \PP$
regardless of whether we started with the \textsc{Jet interpolation
  problem} or the \textsc{Function interpolation problem}.

Let
\begin{equation*}
E = \{ a_k \mid k = 1, \ldots, N \},
\end{equation*}
be an indexation of the set $E$. The first step is to compute the shifted points:
\begin{equation*}
\widetilde{E} = \{ \tilde{a}_k = a_k - D_{a_k}f / M \mid k = 1,
\ldots, N\}.
\end{equation*}
Clearly this requires $O(N)$ work and $O(N)$ storage.

\subsection{Computing $\K_{\ast}$ and $\widehat{\K}$} \label{sec:
  compute Kast and Khat}

With $\widetilde{E}$ in hand, we compute the power diagram
$\K_{\ast} = \PD(\widetilde{E})$ and the dual triangulation
$\widehat{\K} = \DT(\widetilde{E})$. We employ the lifting procedure
via the map $\lambda: \widetilde{E} \rightarrow \R^{d+1}$ described in
Section \ref{sec: pd dt ch}. Once
the points have been lifted, we compute the convex hull of
$\lambda(\widetilde{E})$. To determine the lower hull, we compute a
normal vector for each $d$-dimensional facet of the convex hull, and
orient them so that they are pointing inward. The inward pointing
normals determine the facets of the lower hull (namely, if the
$x_{d+1}$ coordinate of the normal is positive, then the facet is on
the lower hull). We then orthogonally project the facets of the lower
hull onto $\R^d$, under the map $(x_1,\ldots,x_d,x_{d+1}) \mapsto
(x_d,\ldots,x_d)$. This gives us the triangulation $\widehat{\K} =
\DT(\widetilde{E})$. It is initially stored in a data a structure
which lists the $d$-dimensional faces, which are simplices. Each
simplex is uniquely determined by storing the $d+1$ indices $\{ k_i
\}_{i=1}^{d+1}$ which correspond to the vertices $\{ \tilde{a}_{k_i}
\}_{i=1}^{d+1}$ of the simplex. The work is $O(N \log N + N^{\lceil
  d/2 \rceil})$ and the storage is $O(N^{\lceil d/2 \rceil})$. 

We make a pass through the simplices of the triangulation and store,
for each simplex, the $(d-1)$-dimensional facets contained in that simplex, which
we refer to as its children. Proceeding in a top-down fashion, we
store for each $j$-dimensional face $\widehat{S} \in \widehat{\K} = \DT
(\widetilde{E})$ ($0 < j < d$), both its children,
which are the $(j-1)$-dimensional faces of $\DT (\widetilde{E})$
contained in $\widehat{S}$, and its parents, which are the
$(j+1)$-dimensional faces of $\DT (\widetilde{E})$ that contain
$\widehat{S}$. For the vertices $\widetilde{E}$, we store only the
parents of each vertex, since they have no children. Each
$j$-dimensional face ($0 < j \leq d$) has exactly $j+1$
children. Additionally, the number of parents per face is no more than
$N$. Since there are $O(N^{\lceil d/2 \rceil})$ faces, the work and
storage for this step is $O(N^{\lceil d/2 \rceil +1})$.

Via duality, we derive the power diagram $\K_{\ast} =
\PD(\widetilde{E})$ from the triangulation $\widehat{\K} = \DT
(\widetilde{E})$. We first compute the power center of each
$d$-dimensional simplex of the triangulation. This is the point that
is equidistant from each vertex of the simplex with respect to the
power function of that vertex. Consider a
$d$-dimensional simplex with vertices
$\{\tilde{a}_{k_i}\}_{i=1}^{d+1}$. The power center is found by
solving for the $\rho \in \R^d$ satisfying:
\begin{equation*}
\pow (\rho,\tilde{a}_{k_1}) = \pow (\rho,\tilde{a}_{k_i}), \quad i =
2,\ldots d+1.
\end{equation*}
This system of equations is in fact linear, and can be rewritten as:
\begin{equation} \label{eqn: power center linear system}
2 (\tilde{a}_{k_i} - \tilde{a}_{k_1}) \cdot \rho = w(\tilde{a}_{k_1}) -
w(\tilde{a}_{k_i}) + |\tilde{a}_{k_i}|^2 - |\tilde{a}_{k_1}|^2, \quad
i = 2, \ldots, d+1.
\end{equation}
Equation \eqref{eqn: power center linear system} is a system of $d$
linear equations and $d$ unknowns; thus we
have a unique solution for $\rho$ so long as the simplex is not
degenerate, which is the power center. We make a pass over the
$d$-dimensional simplices of $\widehat{\K} = \DT(\widetilde{E})$ and
store the list of corresponding power centers $\{ \rho_{\ell} \}_{\ell
  \geq 1}$. The additional work and storage is $O(N^{\lceil d/2 \rceil})$, which
is the number of $d$-dimensional simplices of $\DT (\widetilde{E})$.

As discussed in Section \ref{sec: pd dt ch},
the power centers are vertices of the power diagram. The remainder of
the power diagram is determined using the duality between $\PD
(\widetilde{E})$ and $\DT (\widetilde{E})$. Edges ($1$-dimensional
faces) are determined using the $(d-1)$-dimensional faces of
$\widehat{\K} = \DT (\widetilde{E})$. Each $(d-1)$-dimensional facet of
$\DT (\widetilde{E})$ has one or two parents, which are
$d$-dimensional simplices. Those facets with two parents are dual to an
edge in $\K_{\ast} = \PD (\widetilde{E})$, which runs between the two
power centers corresponding to the two parent simplices. A facet $\widehat{S}$
of $\DT (\widetilde{E})$ with only one parent lies on the exterior of $\DT
(\widetilde{E})$; the edge of $\PD (\widetilde{E})$ dual to this facet
lies on an unbounded power cell, and thus has infinite length. It
originates at the power center corresponding to the single parent
simplex of $\widehat{S}$, and its direction (computed using $C(d)$
work) is perpendicular to $\widehat{S}$. The edge is stored by
generating a synthetic point along it, and storing this point in
addition to the power center at which the edge originates. These
synthetic points are marked as such, so they are distinguishable from
the power centers. The additional storage and cost for computing the
edges of $\PD (\widetilde{E})$ is $O(N^{\lceil d/2 \rceil})$. 

Higher dimensional faces of $\K_{\ast} = \PD (\widetilde{E})$ are
computed recursively. Suppose the vertices of all $j$-dimensional
faces of $\PD (\widetilde{E})$ have been stored, for some $j \geq
1$. Let $S_{\ast} \in \PD (\widetilde{E})$ be a $(j+1)$-dimensional
face, and let $\widehat{S} \in \DT (\widetilde{E})$ be its
corresponding dual face in the dual triangulation. The parents of
$\widehat{S}$ (stored earlier) correspond to the $j$-dimensional faces of $\PD
(\widetilde{E})$ contained in $S_{\ast}$. The vertices for these faces
have already been computed and stored; their union is the set of vertices of
$S_{\ast}$. Since the number of $j$-dimensional faces of $\PD
(\widetilde{E})$ is $O(N^{\lceil d/2 \rceil})$ for each $0 \leq j <
d$, the work and storage for this step is $O(N^{d+1})$. 

\subsection{Computing the $S_C$ points}

With $\K_{\ast} = \PD (\widetilde{E})$ and $\widehat{\K} = \DT
(\widetilde{E})$ computed, the remainder of the
one-time work is devoted to computing the $S_C$ points and the final
cells $\{T_S\}_{S \in \K}$. We begin with the former.

Let $\widehat{S} \in \widehat{\K}$ be a $j$-dimensional face of the
triangulation, and let $S_{\ast} \in \K_{\ast}$ be its dual
$(d-j)$-dimensional face in the power diagram. Recall that $S_H$ is an
affine space containing $\widehat{S}$, and note that $S_E$ is also an
affine space (this follows from \cite[Lemma 1,
p. 142]{wells:diffFuncLipDeriv1973}), which contains $S_{\ast}$. These
two spaces are orthogonal, and $S_C$ is the single point
of intersection, i.e., $S_C = S_H \cap S_E$. 

When $j=0$ or $j=d$, finding $S_C$ is simple. When $j=0$, $\widehat{S}
= \tilde{a} \in \widetilde{E}$ for some $a \in E$, and $S_C =
\tilde{a}$. Similarly, when $j = d$, $S_{\ast}$ is a power center, and
$S_C$ is this power center. 

When $0 < j < d$, let $\{\tilde{a}_{k_i}\}_{i=1}^{j+1}$ be the
vertices of the simplex $\widehat{S}$. Compute a set of vectors
$\{v_i\}_{i=1}^j$, where $v_i = \tilde{a}_{k_1} -
\tilde{a}_{k_{i+1}}$. Similarly, let $\{\rho_{\ell_k}\}_{k \geq 1}$ be the vertices of
$S_{\ast}$, which are power centers (note the number can vary and
depends on $S_{\ast}$, but is bounded by $O(N^{\lceil d/2
  \rceil})$). Select $d-j$ power centers
$\{\rho_{\ell_{k_i}}\}_{i=1}^{d-j}$ from amongst $\{ \rho_{\ell_k} \}_{k
  \geq 2}$ such that $\{ w_i \}_{i=1}^{d-j}$, $w_i = \rho_{\ell_1} -
\rho_{\ell_{k_i}}$ are linearly independent. $S_C$ is the unique point
for which there exists $\{\gamma_i\}_{i=1}^j \subset \R$ and
$\{\beta_i\}_{i=1}^{d-j} \subset \R$ such that
\begin{equation*}
S_C = \tilde{a}_{k_1} + \sum_{i=1}^j \gamma_i v_i = \rho_{\ell_1} +
\sum_{i=1}^{d-j} \beta_i w_i,
\end{equation*}
The work for solving for $\{\gamma_i\}_{i=1}^j$ and
$\{\beta_i\}_{i=1}^{d-j}$ depends only on $d$. Obtaining the $d-j$
linearly independent vectors $\{ w_i \}_{i=1}^{d-j}$ costs at most
$O(N^{\lceil d/2 \rceil})$ per $S \in \K$, and thus the total work is
bounded by $O(N^{d+1})$. 

With $S_C$ computed, we also compute $d_S(S_C)$ and store it away for
use in the query work. The cost is $O(N^{\lceil d/2 \rceil})$. The
additional storage needed throughout is no more than $O(N^{\lceil d/2
  \rceil})$.

\subsection{Computing the $T_S$ cells} \label{sec: computing TS cells}

To compute the cells $\{T_S\}_{S \in \K}$, we find matrices $A_S$
and vectors $b_S$ such that
\begin{equation*}
x \in T_S \quad \text{if and only if} \quad A_S x \leq b_S.
\end{equation*}
The pair $(A_S,b_S)$ determines the bounding hyperplanes of
$T_S$. These hyperplanes are determined by $\widehat{S} \in
\widehat{\K} = \DT (\widetilde{E})$ and
$S_{\ast} \in \K_{\ast} = \PD (\widetilde{E})$. Suppose that
$\widehat{S}$ is $j$-dimensional, and $S_{\ast}$ is
$(d-j)$-dimensional. Let $U \subset S_{\ast}$ be a
$(d-j-1)$-dimensional face of $\PD (\widetilde{E})$, contained in
$S_{\ast}$. Then the hyperplane containing $\frac{1}{2}(\widehat{S} +
U)$ is a bounding hyperplane of $T_S$. Similarly, if $V \subset
\widehat{S}$ is a $(j-1)$-dimensional face of $\DT (\widetilde{E})$
contained in $\widehat{S}$, then $\frac{1}{2}(S_{\ast} + V)$ also
determines a bounding hyperplane of $T_S$. Doing this over all such
$(d-j-1)$-dimensional faces $U \subset S_{\ast}$
and $(j-1)$-dimensional faces $V \subset \widehat{S}$ gives the set of
bounding hyperplanes of $T_S$. 

We compute the pair $(A_S, b_S)$ as follows. Let $\{ \tilde{a}_{k_i}
\}_{i=1}^{j+1}$ be the vertices of $\widehat{S}$ and let $\{
\rho_{\ell_m} \}_{m \geq 1}$ be the vertices of $S_{\ast}$, both of
which have been stored from the calculation in Section \ref{sec:
  compute Kast and Khat}. The vertices of $T_S$ are $\left\{
  \frac{1}{2} (\tilde{a}_{k_i} + \rho_{\ell_m}) \right\}_{i, m}$. We
compute the mean vector $\mu_S$ of this vertex set and store it
away. We will use $\mu_S$ to center $T_S$ at the origin, for the
purpose of computing $(A_S, b_S)$. The work and storage is no more
than $O (N^{\lceil d/2 \rceil})$ per $S \in \K$.

First fix $S_{\ast}$ and consider the $(j-1)$-dimensional faces $V \in
\widehat{\K} = \DT (\widetilde{E})$ contained in $\widehat{S}$. These are the
children of $\widehat{S}$, which have been stored previously. Let $\{
\tilde{a}_{k_n}\}_{n=1}^j$ be the vertices of $V$. We compute:
\begin{equation} \label{eqn: vertex set of hyperplane}
\left\{ \frac{1}{2} (\tilde{a}_{k_n} + \rho_{\ell_m}) - \mu_S
\right\}_{n, m},
\end{equation}
which are the vertices of $T_S$ corresponding to $\frac{1}{2}(S_{\ast}
+ V)$, centered at the origin. Collect $d$ (linearly independent)
vectors from \eqref{eqn: vertex set of hyperplane}, and store them as
the rows of a $d \times d$ matrix $M_V$. The solution
$\alpha \in \R^d$ to $M_V \alpha =
\mathbf{1}$ gives the hyperplane $(\alpha, 1)$ relative to
the origin. To shift it back to the original location of $T_S$, we
compute $b = 1 + \alpha \cdot \mu_S$. The vector $\alpha$ is one row
of $A_S$, and the scalar $b$ is the corresponding entry in $b_S$. The
work and storage is $O(N^{\lceil d/2 \rceil})$ per $S \in \K$. After
passing over all $S \in \K$, the total work for this step is
$O(N^{d+1})$ and the maximum storage needed at any given time is
$O(N^{\lceil d/2 \rceil})$.

The remaining hyperplanes are obtained analogously. Fix $\widehat{S}$
and consider the $(d-j-1)$-dimensional faces $U \in \PD
(\widetilde{E})$ contained in $S_{\ast}$. These faces are obtained by
going through the list of parents of $\widehat{S}$, stored previously
in Section \ref{sec: compute Kast and Khat}, and then looking up the dual
face stored for $\PD (\widetilde{E})$. The remainder of the
calculation is the same as in the previous paragraph. The work and
storage is $O(N^{\lceil d/2 \rceil +1})$ per $S \in \K$, owing to the fact that each
$\widehat{S}$ can have $O(N)$ parents. The total work is therefore
$O(N^{d+2})$ and the total storage is $O(N^{\lceil d/2 \rceil + 1})$. 

This concludes the one time work. Summing over the complexity of all
calculations described in Section \ref{sec: one time work}, the total
work is $O(N^{d+2})$ and the storage is $O(N^{\lceil d/2 \rceil +
  1})$. Figure \ref{fig: TS cells} contains an image of the final cellular
decomposition $\{ T_S \}_{S \in \K}$ in two dimensions.

\begin{figure}
\center
\frame{\includegraphics[width=3.0in]{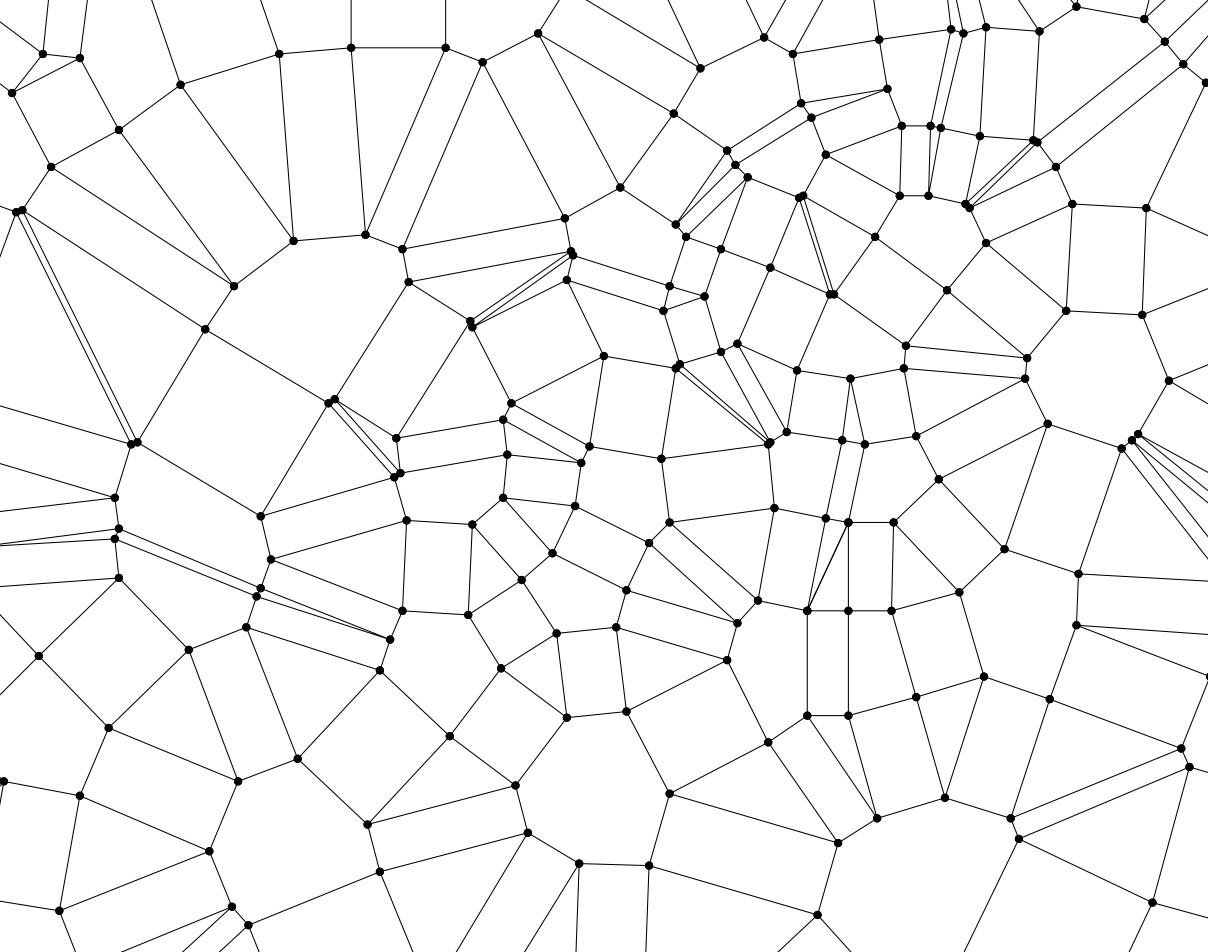}}
\caption{The final cellular decomposition consisting of the $T_S$
  cells derived from the power diagram and triangulation in Figure
  \ref{fig: PD and DT}.}
\label{fig: TS cells}
\end{figure}

\section{Query work} \label{sec: query work}

We now consider the query work. At this point the one
time work is complete, and the algorithm is ready to accept a query
point $x \in \R^d$. The first step is to determine the cell $T_S$ such
that $x \in T_S$. Then we evaluate $F_S(x)$ and $\nabla F_S(x)$. 

\subsection{Determining $T_S$}

Recall that as part of the one time work, for each cell $T_S$ we have
stored a matrix $A_S$ and a vector $b_S$ such that 
\begin{equation*}
x \in T_S \quad \text{if and only if} \quad A_Sx \leq b_S.
\end{equation*}
Therefore, the simplest way to determine the set $T_S$ such that $x
\in T_S$ is to check if $A_Sx \leq b_S$ for each $S \in \K$. There are
$O(N^{\lceil d/2 \rceil})$ sets in $\K$, and the number of rows in any
matrix $A_S$ is bounded from above by $d \cdot N$. Thus, using this
approach, the query work is $O (N^{\lceil d/2 \rceil + 1})$. 

An alternative is given in \cite{fuchs:polytopicPointLoc2010}, which
describes an algorithm for efficient point location in general
polytopic data sets. Since $\{T_S\}_{S \in \K}$ is polytopic (aside
from the unbounded regions, but these cells can be accounted for), we can
utilize this point location algorithm for our query
work. Indeed, the point location algorithm of
\cite{fuchs:polytopicPointLoc2010} requires only that each polytopic set be
described via a matrix $A$ and a vector $b$, just as we have done in
Section \ref{sec: computing TS cells}. 

Applied to our particular data structure, the algorithm results in a
tree structure over the cellular decomposition $\{T_S\}_{S
  \in \K}$. More precisely, the authors build a binary tree in which
each node corresponds to a subspace of $\R^d$. The nodes are split via
a hyperplane, with the left child corresponding to the part of the
subspace lying to the ``left'' of the hyperplane, and the right child
corresponding to the subspace lying to the ``right'' of the
hyperplane. The root of the tree is $\R^d$, and the leaves of the tree
are the cells $\{T_S\}_{S \in \K}$. 

One way to build such a tree is to use the bounding hyperplanes of the
cells $\{T_S\}_{S \in \K}$, and to optimize your selection of the
hyperplanes in some fashion. This is proposed in \cite{tondel:PWA2003},
and in many cases will yield a balanced tree that can evaluate queries
in $O(\log N)$ time. As discussed in
\cite{fuchs:polytopicPointLoc2010}, there is no guarantee though and
some cellular arrangements will yield unbalanced trees via this method. The
alternate algorithm proposed in \cite{fuchs:polytopicPointLoc2010}
utilizes splitting hyperplanes that are not necessarily bounding
hyperplanes of $\{T_S\}_{S \in \K}$, but can ensure that the tree
is balanced. These hyperplanes are computed by solving a certain
optimization problem, that is hard to analyze precisely
so as to determine the additional one time work. Nevertheless, the
benefit to the query algorithm is clear, as it would guarantee that
the query work is $O(\log N)$. 

\subsection{Evaluating the interpolant}

Now we must compute $F(x)$ and $\nabla F(x)$. Recall that $T_S =
\frac{1}{2}(\widehat{S} + S_{\ast})$, where 
$\widehat{S} \in \widehat{\K} = \DT(\widetilde{E})$ is a face in the
triangulation, and $S_{\ast} \in \K_{\ast} = \PD(\widetilde{E})$ is
the dual (possibly unbounded) face that is part of the power diagram. Each
point $x \in T_S$ has a unique representation $x = \frac{1}{2}(y+z)$,
  where $y \in \widehat{S}$ and $z \in S_{\ast}$. Additionally, recall
  that for $x \in T_S$,
\begin{equation*}
F(x) = F_S(x) = d_S(S_C) + \frac{M}{2}d(x,S_H)^2 -
\frac{M}{2}d(x,S_E)^2, \quad x \in T_S.
\end{equation*}
Since $\distset(x,S_H) = \frac{1}{2}\distset(z,S_H) = \frac{1}{2}|z-S_C|$ and
$\distset(x,S_E) = \frac{1}{2}\distset(y,S_E) = \frac{1}{2}|y-S_C|$,
we can rewrite $F_S$ as:
\begin{equation*}
F_S(x) = d_S(S_C) + \frac{M}{8}|z-S_C|^2 - \frac{M}{8}|y-S_C|^2, \quad
x = \frac{1}{2}(y+z) \in T_S.
\end{equation*}
From Section \ref{sec: wells} we also know that the gradient $\nabla F_S(x)$
can be written in terms of $y$ and $z$:
\begin{equation*}
\nabla F_S(x) = \frac{M}{2}(z-y), \quad x = \frac{1}{2}(y+z) \in T_S.
\end{equation*}
Since the one time work stores $S_C$ and $d_S(S_C)$, to return $F(x)$
and $\nabla F(x)$ we must find the $y \in \widehat{S}$ and $z
\in S_{\ast}$ such that $x = \frac{1}{2}(y+z)$. This is accomplished
by projecting $x$ onto $S_H$ and $S_E$, and using the positions of the
projected points relative to $S_C$ to find $y$ and $z$. The amount of
work only depends on the dimension. 

\section{Numerical simulations} \label{sec: numerical simulations}

We report the run times and complexity of numerical
simulations\footnote{Using the code available at:
  \href{https://github.com/matthew-hirn/C-1-1-Interpolation}{\texttt{https://github.com/matthew-hirn/C-1-1-Interpolation}}}
in order to give the reader an idea of the real time cost of computation. All
computations were computed on an Apple iMac desktop computer with 32
GB of RAM and a 4 GHz Intel Core i7 processor. The unit of time is
seconds. The set $E$, consisting of $N$ points in $\R^d$, was
uniformly randomly selected from the cube $[0,N^{2/d}]^d$. The
function values and partial derivatives were uniformly randomly
selected from the set $[-1.1,-0.9] \cup [0.9, 1.1]$. Query work run
times are the average of $2^{10}$ queries uniformly randomly selected
from the cube $[-1,N^{2/d}+1]^d$.

Figures \ref{fig: d=2}, \ref{fig: d=3} and \ref{fig: d=4} show $\log_2$--$\log_2$ plots of
the one time work, and its three primary components (computing
$\Gamma^1(P; E)$, computing $\DT (\widetilde{E})$ / $\PD
(\widetilde{E})$, and computing the cells $\{ T_S \}_{S \in \K}$), as
a function of $N$ for dimensions $d = 2, 3, 4$. The $\log_2$--$\log_2$
plots of each of the components grow linearly, indicating that the
work scales as $O(N^{\alpha})$ for some $\alpha > 0$. The computation
of $\Gamma^1(P; E)$ we know grows as $O(N^2)$. For $d = 2, 3$, this
component has the steepest slope, indicating that in practice the
computation of the power diagram and dual triangulation, in addition
to the cells $\{ T_S \}_{S \in \K}$, grows as $O (N^{\alpha})$ for
some $\alpha < 2$. For $d=4$ however, the computation of $\DT
(\widetilde{E})$ / $\PD (\widetilde{E})$ appears to grow nearly at the
rate $C \cdot N^2$, which is the same asymptotically as the $C' \cdot
N^2$ growth for the computation of $\Gamma^1(P; E)$. However, the
constant factor $C$ is significantly larger than $C'$, which affects
the practicality of the computation. We can conclude that for $d \geq
4$, the computation of the power diagram will grow at least as
$O(N^2)$, with a large constant factor independent of $N$. Figure
\ref{fig: numK} shows the total number of cells $\{ T_S \}_{S \in \K}$
(equivalently the number of faces of $\DT
(\widetilde{E})$ / $\PD (\widetilde{E})$) as a function of $N$, for $d
= 2, 3, 4$. In numerical
simulations, the query work (not plotted) grew proportionally to the
number of these cells.

\begin{figure}[h]
\center
\subfigure[$d=2$]{
\includegraphics[width=2.25in]{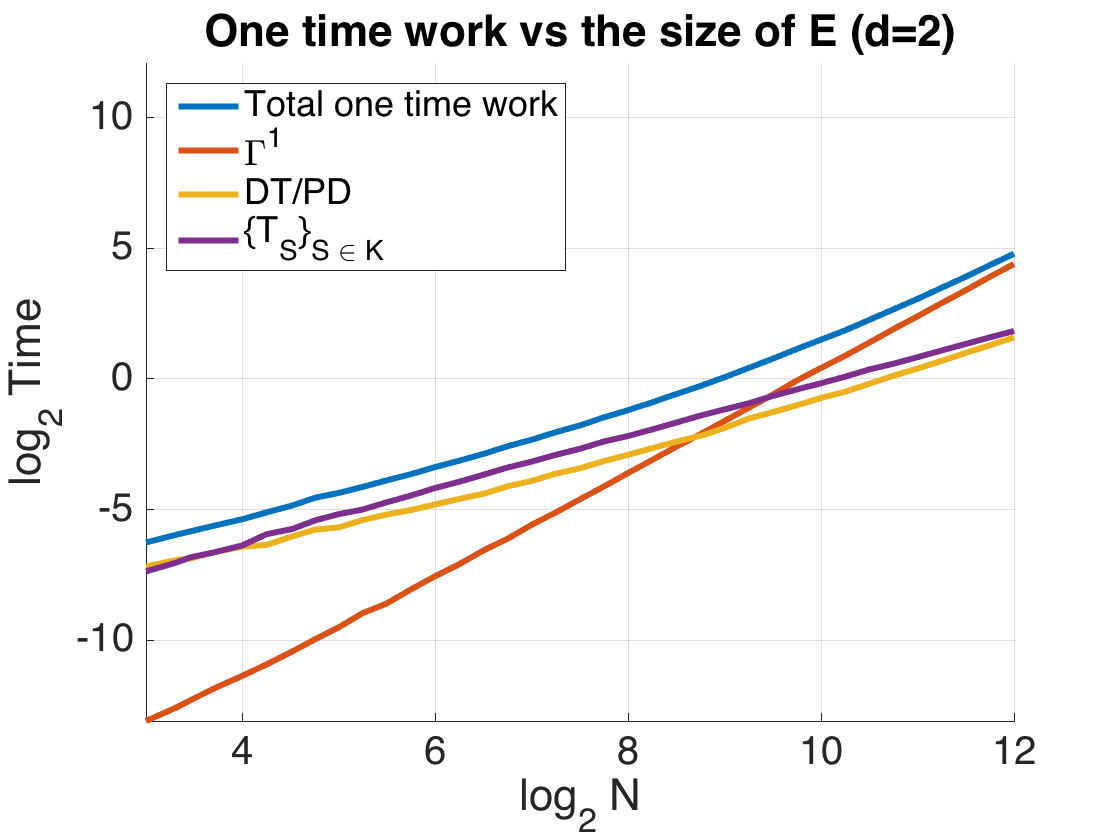}
\label{fig: d=2}
}
\subfigure[$d=3$]{
\includegraphics[width=2.25in]{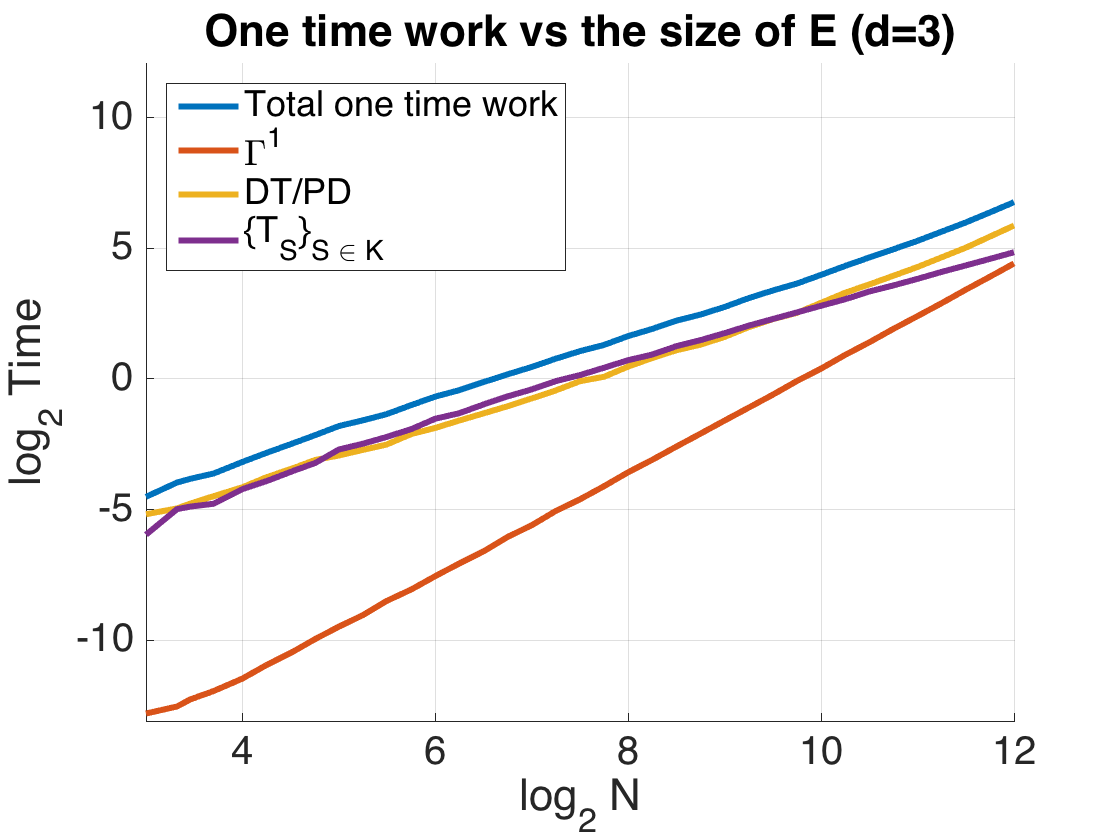}
\label{fig: d=3}
}
\subfigure[$d=4$]{
\includegraphics[width=2.25in]{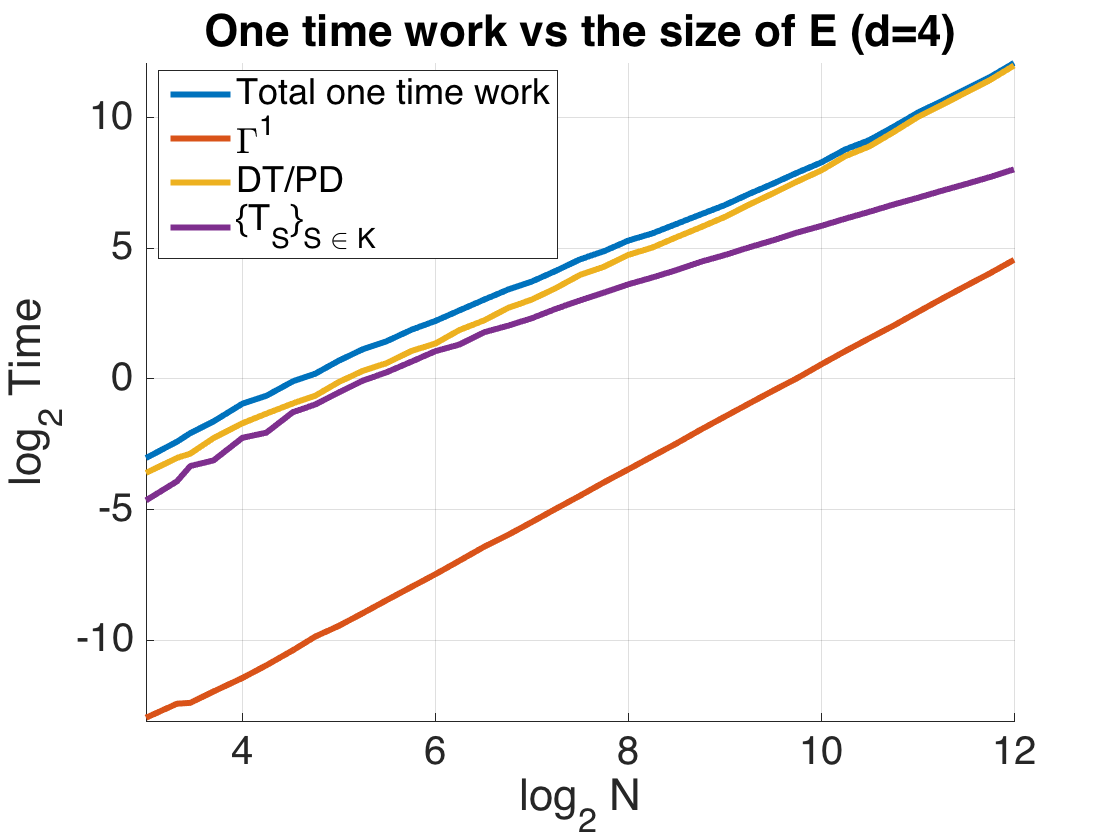}
\label{fig: d=4}
}
\subfigure[$\#(\K)$]{
\includegraphics[width=2.25in]{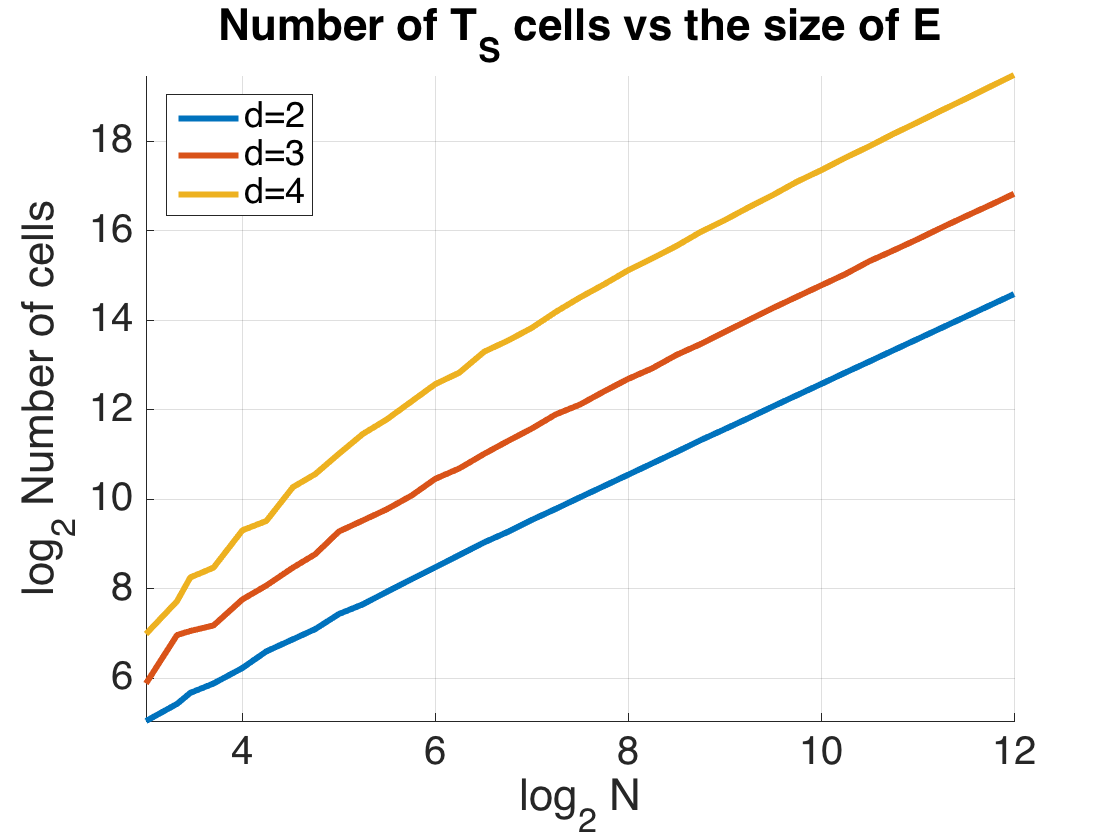}
\label{fig: numK}
}
\caption{One time work and number of cells $\{T_S\}_{S \in \K}$ versus $N
  = \#(E)$ on a $\log_2$--$\log_2$ scale, for dimensions $d = 2, 3, 4$.}
\label{fig: vary N fix d}
\end{figure}

The accuracy and stability of the algorithm were tested by computing
$|F(a) - f_a|$ and $|\partial_{x_i} F(a) - D_a f(i)|$ for each $a \in
E$ and $i = 1, \ldots d$. The algorithm reported an error whenever
the absolute difference was more than $10^{-10}$. Across all values of $N$
tested, the number of errors for each dimension was:
\begin{itemize}[itemsep=1pt]

\item
$d=2$: Function value errors: $0 / 25701$; Partial derivative errors: $9/51402$.

\item
$d=3$: Function value errors: $10 / 25701$; Partial derivative errors: $30/77103$.

\item
$d = 4$: Function value errors: $3 / 25701$; Partial derivative
errors: $12/102804$.

\end{itemize}
Thus across dimensions $d = 2, 3, 4$, the interpolation algorithm made
an error only 0.0208\% of the time. 

We plot the logarithm of the one time work against the dimension
$d$, for a fixed size $E$, with $N = 16$, $18$ and $20$ in Figures
\ref{fig: N=16}, \ref{fig: N=18} and \ref{fig: N=20},
respectively. Figure \ref{fig: numK against d} plots $\log_2 \#(\K)$
against the dimension $d$. Notice that all four graphs are linear. If numerically
$\#(\K) = O(N^{C_{\K} d})$ holds, then we should be able to solve for a
consistent value of $C_{\K}$ from the three plots in Figure \ref{fig: numK
  against d}, using $\frac{\log_2 \#(\K)}{d} = C_{\K} \log_2N$. We estimate
the slopes $\frac{\log_2 \#(\K)}{d}$ using a least squares fit, and
then solve for $C_{\K}$, obtaining $C_{\K} = 0.3241, 0.3149, 0.3117$ for $N =
16, 18, 20$, respectively. Thus, for these experiments, we obtain
$C_{\K} \approx 0.3169$. For the one time work, the computation of $\PD
(\widetilde{E})$ / $\DT (\widetilde{E})$ dominates as $d$
increases. We have the similar hypothesis that the cost of this
computation is $O (N^{C_{\PD} d})$, and for $N = 16, 18, 20$ estimate
that $C_{\PD} =  0.5389, 0.5240, 0.5433$, respectively. Thus $C_{\PD}
\approx 0.5354$ for this collection of numerical simulations. Recall from
Section \ref{sec: compute Kast and Khat} that the theoretical bound
for the number of cells is $\#(\K) = O(N^{\lceil d/2 \rceil})$, while
the cost of computing the relevant data structures to hold $\DT
(\widetilde{E})$ / $\PD (\widetilde{E})$ was $O(N^{d+1})$. From this
limited collection of numerical experiments, one might hypothesize for
uniformly sampled data an expected number of cells growing as $\#(\K)
\approx O(N^{d/3})$ and a one time work cost that is approximately
$O(N^{d/2})$.

Figure \ref{fig: interpolant pic} contains a plot of an
extension colored according to the cells $\{T_S\}_{S \in \K}$. 

\begin{figure}[h]
\center
\subfigure[$N=16$]{
\includegraphics[width=2.25in]{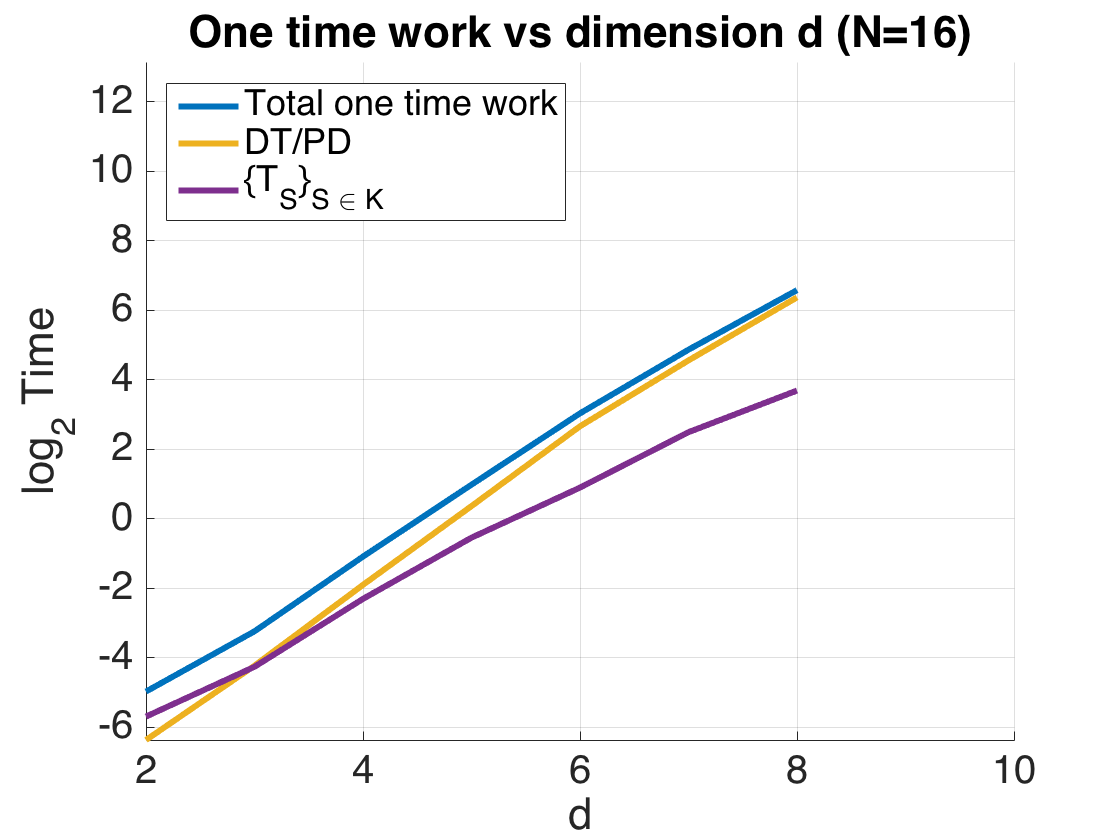}
\label{fig: N=16}
}
\subfigure[$N=18$]{
\includegraphics[width=2.25in]{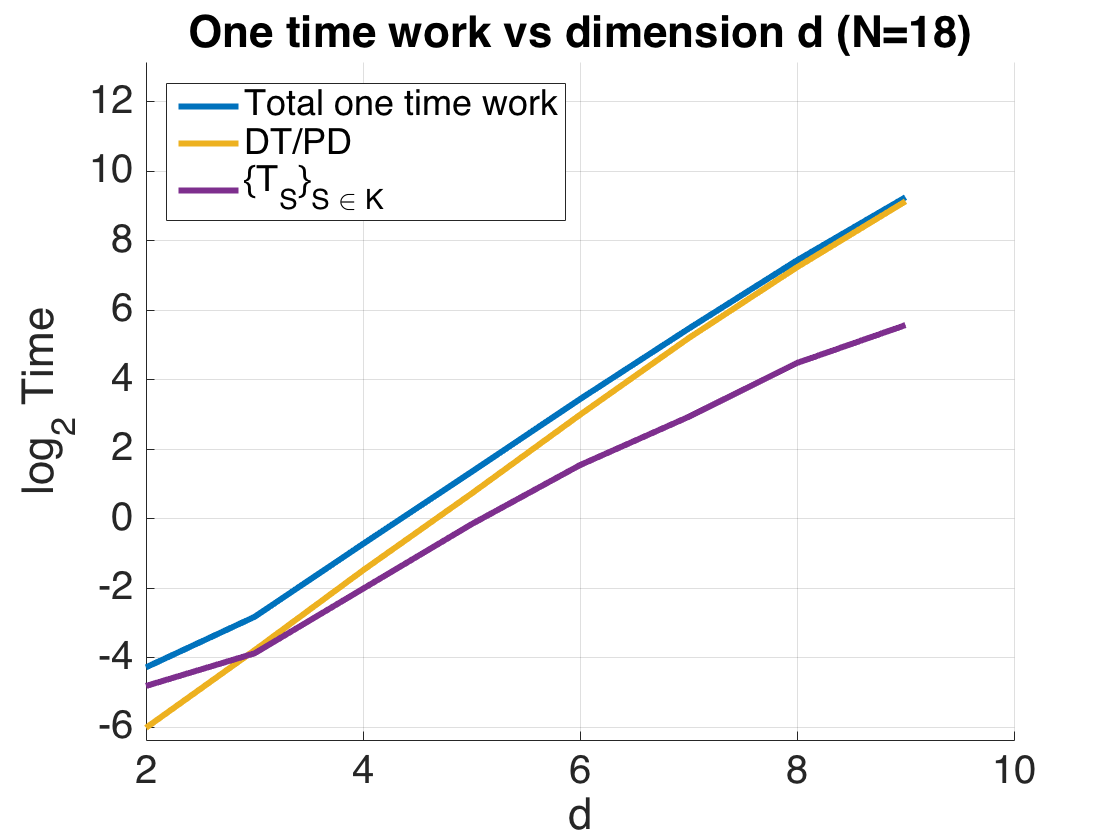}
\label{fig: N=18}
}
\subfigure[$N=20$]{
\includegraphics[width=2.25in]{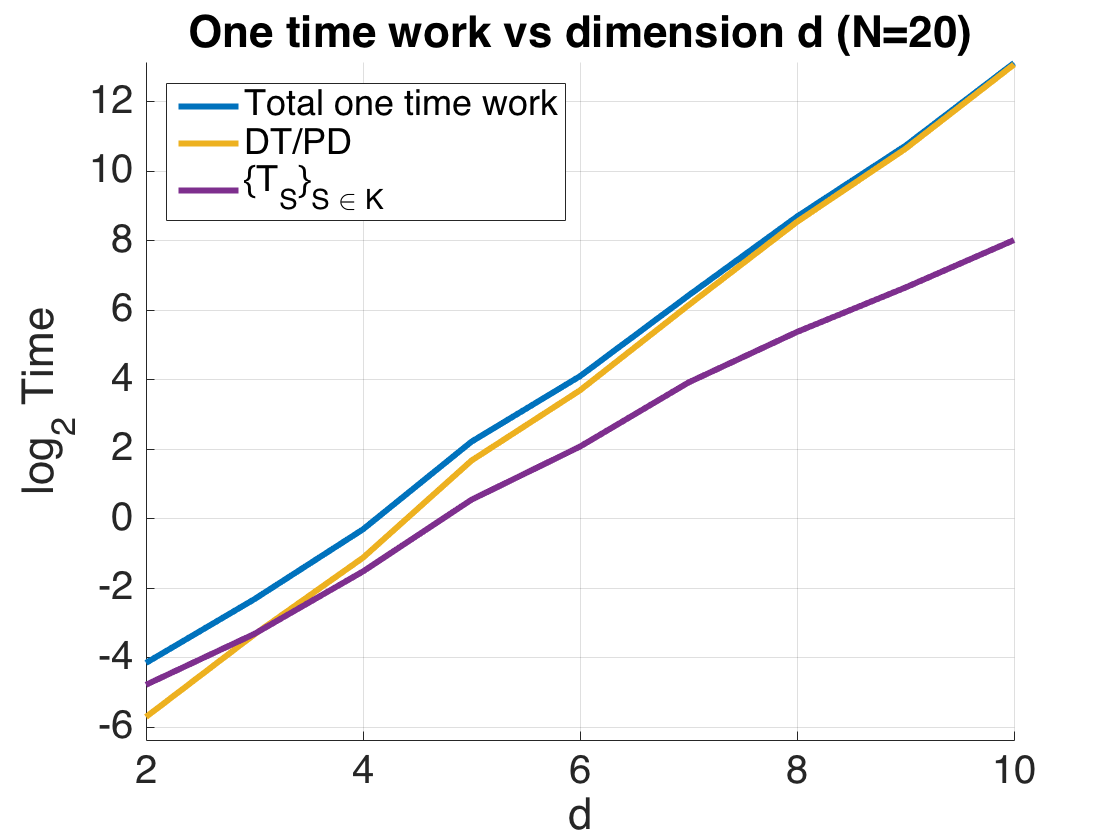}
\label{fig: N=20}
}
\subfigure[$\#(\K)$]{
\includegraphics[width=2.25in]{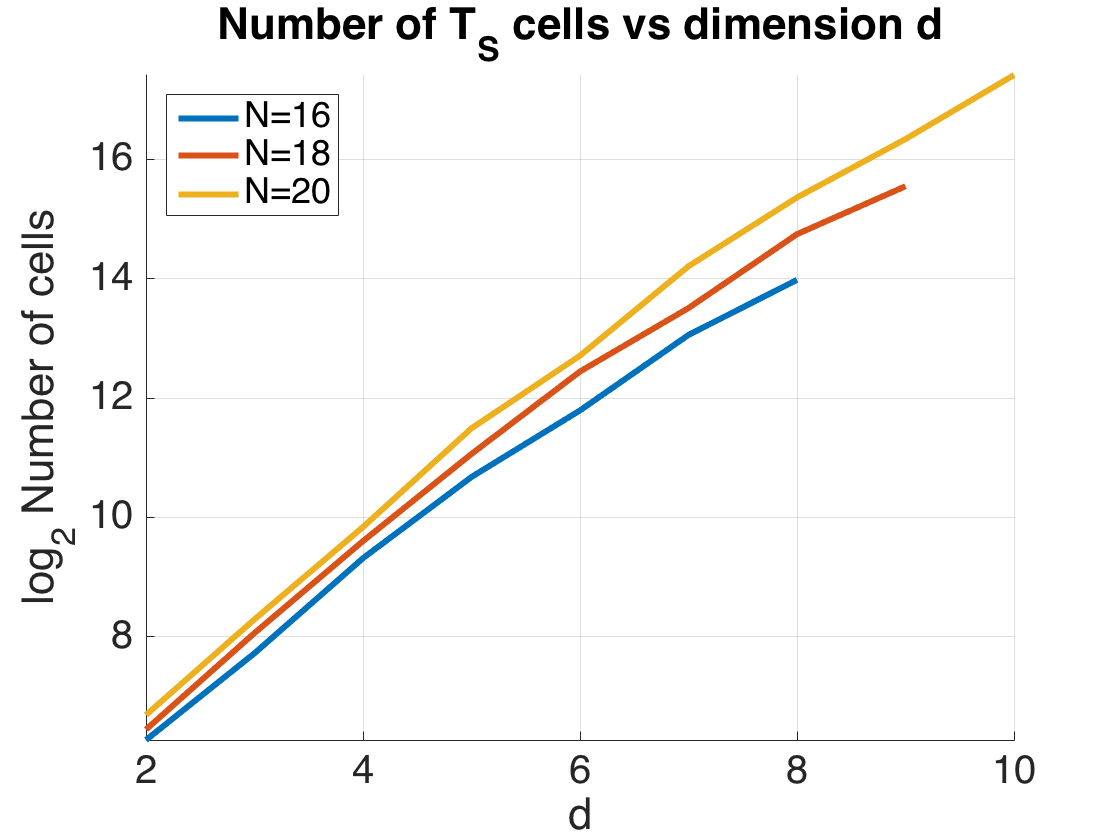}
\label{fig: numK against d}
}
\caption{The logarithm (base $2$) of the one time work and $\log_2 \# \{T_S\}_{S
    \in \K}$ versus the dimension $d$, for $N = 16, 18, 20$.}
\end{figure} 

\begin{figure}[h]
\center
\includegraphics[width=3.5in]{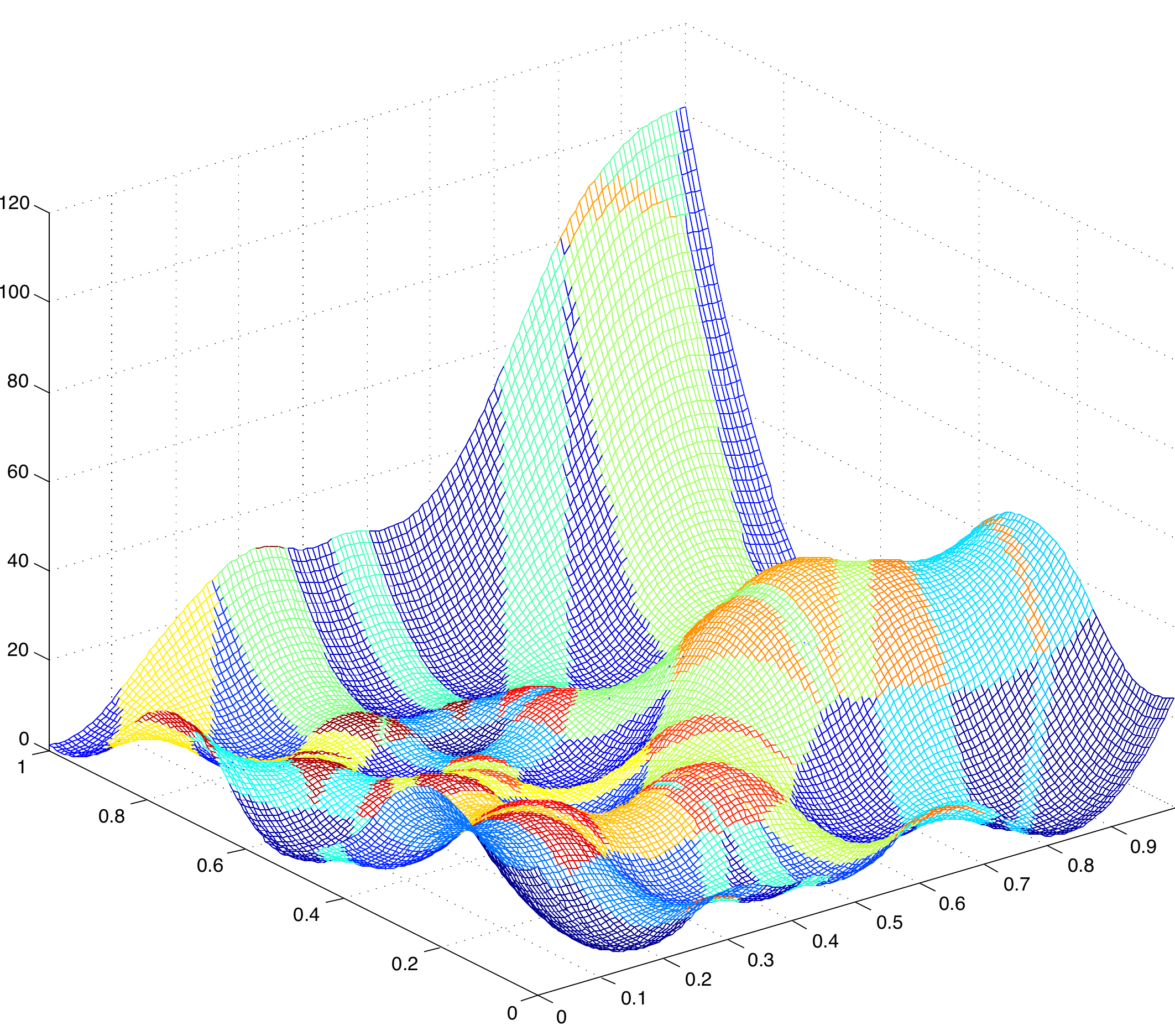}
\caption{Interpolant $F \in C^{1,1}(\R^2)$ colored by the cells $\{T_S\}_{S \in \K}$.}
\label{fig: interpolant pic}
\end{figure}

\section{Conclusion}

We introduced an efficient, practical algorithm for computing
interpolations of data, consisting of a set $E \subset \R^d$ and
either a $1$-field $P: E \rightarrow \PP$ or a function $f: E
\rightarrow \R$, with functions $F \in C^{1,1}(\R^d)$ such that $\Lip
(\nabla F)$ is within a dimensionless factor of being minimal. Amongst
Whitney type interpolation algorithms, it is the first algorithm to be
implemented on a computer, thus moving the development of these
algorithms from theoretical constructs to proof of concept. Numerical
experiments indicate that the algorithm run time is reasonable up to
dimension $d = 10$ for small sets $E$, as well as for larger sized
initial data in small dimensions.

Fundamentally, the algorithm is built upon two key components: (1) the
theoretical results of Wells \cite{wells:diffFuncLipDeriv1973} and Le
Gruyer \cite{legruyer:minLipschitzExt2009}, giving a construction of
the interpolant $F$ and a closed form solution for the minimal value
of $\Lip (\nabla F)$, respectively; and (2) a collection of notions
and algorithms from computational geometry, including the well
separated pairs decomposition \cite{callahan:wspd1995}, power diagrams, triangulations, and
convex hulls, which are used to compute the relevant values,
underlying structures, and ultimately the interpolant $F$. 

These results are opening new mathematical avenues
related to Whitney's extension theorem, both pure and applied. High
dimensional data is being collected on an unprecedented scale; thus
efficiency in both the size of $E$ and the dimension $d$ is
needed. Can the efficiency of the algorithm be
improved, while maintaining its precision? Notions related to
approximate Voronoi diagrams \cite{Har-Peled:approxVoronoi2001} may be
of use here, if the interpolant construction is stable. Looking
further ahead, an algorithm of this type for $C^{2,1}(\R^d)$ would first
require results analogous to those of Wells and Le Gruyer, pertaining
to precise formulations of the best Whitney's constant. Positive
results along these lines and others have the potential to push
Whitney interpolation algorithms beyond proof of concept and into the
applied sciences.

\section{Acknowledgements}

M.H. would like to thank Charles Fefferman for
introducing him to the problem as well as for several helpful
conservations. He would also like to thank Erwan Le Gruyer and
Hariharan Narayanan for numerous insightful discussions.

All three authors would like to thank the anonymous referee whose
numerous corrections, suggestions, and insights greatly improved the
manuscript.

\begin{appendices}

\section{Convex optimization} \label{sec: convex optimization}

Everything in this Appendix can be found in
\cite{boyd:convexOptimization2004}; we have summarized the parts
relevant to Section \ref{sec: function interpolation problem} to serve
as a convenient reference.

\subsection{Self-concordant functions}

A convex function $h: \R \rightarrow \R$ is said to be \emph{self-concordant}
if 
\begin{equation*}
|h^{(3)}(x)| \leq 2 h^{(2)}(x)^{3/2}, \quad \forall \, x \in \R.
\end{equation*}
If $h: \R^n \rightarrow \R$, then we say it is self-concordant if
$\tilde{h}(t) = h(x + tv)$ is self-concordant as a function of $t \in \R$ for
all $x,v \in \R^n$. 

Self-concordant functions were first introduced by Nesterov and
Nemirovski \cite{nesterov:interiorPoint1994}. They are particularly
useful in convex optimization since they form a class of
functions for which one can rigorously analyze the complexity of
Newton's method.

\subsection{Unconstrained optimization} \label{sec: unconstrained optimization}

An unconstrained convex optimization problem is one of the form:
\begin{equation} \label{eqn: unconstrained convex opt}
\text{minimize} \quad h(x),
\end{equation}
where $h: \R^n \rightarrow \R$ is convex. Unconstrained convex
optimization problems can be solved via any number of descent
methods. \\

\begin{mdframed}[nobreak=true,frametitle={\small General descent algorithm}]
\small
\setlength{\parskip}{0pt}
\texttt{Input: A starting point $x$.\\
\\
Repeat:
\begin{enumerate}[nolistsep,noitemsep]
\item
Determine a descent direction $\Delta x$.
\item
Line search: Choose a step size $t > 0$.
\item
Update: $x \mapsto x + t \Delta x$.
\end{enumerate}
Until stopping criterion is satisfied.
}
\end{mdframed}
\vskip 10pt

There are several ways to determine the descent direction. We focus on
Newton's method \cite[p. 484]{boyd:convexOptimization2004}, where
$\Delta x$ is given by the Newton step:
\begin{equation*}
\Delta x = - \nabla^2 h(x)^{-1} \nabla h(x).
\end{equation*}
The line search is performed using a backtracking line search
\cite[p. 464]{boyd:convexOptimization2004}. \\

\begin{mdframed}[nobreak=true,frametitle={\small Backtracking line
    search}]
\small
\setlength{\parskip}{0pt}
\texttt{Input: A descent direction $\Delta x$ for $h$ at $x$, $\alpha
  \in (0, 0.5)$, $\beta \in (0,1)$.\\
\\
$t = 1$ \\
While $h(x + t\Delta x) > h(x) + \alpha t \nabla h(x) \cdot \Delta x$,
$t \mapsto \beta t$.
}
\end{mdframed}
\vskip 10pt

Each iteration of Newton's method is often times referred to as a Newton step,
even though it entails both computing the Newton step and performing the
line search. 

\subsection{Constrained optimization}

A constrained convex optimization problem is of the form:
\begin{align}
\text{minimize} &\quad h_0(x) \label{eqn: convex opt problem}\\
\text{subject to} &\quad h_i(x) \leq 0, \quad i=1, \ldots, m, \nonumber
\end{align}
where the functions $h_0,\ldots,h_m: \R^n \rightarrow \R$ are
convex. The function $h_0$ is the objective, while the functions
$h_1,\ldots,h_m$ are the constraints. One can solve \eqref{eqn: convex
  opt problem} using interior point methods. 

Define $\phi: \R^n \rightarrow \R$ as
\begin{equation*}
\phi(x) = -\sum_{i=1}^m \log (-h_i(x)).
\end{equation*}
The following unconstrained optimization problem is closely related to
\eqref{eqn: convex opt problem},
\begin{equation} \label{eqn: log barrier}
\text{minimize} \quad th_0(x) + \phi(x),
\end{equation}
where $t > 0$. Indeed, \eqref{eqn: log barrier} is an approximation
for our original constrained convex optimization problem, and as $t
\rightarrow \infty$ the two problems become equivalent. 

The barrier method \cite[p. 569]{boyd:convexOptimization2004} solves
\eqref{eqn: convex opt problem} by iteratively solving \eqref{eqn: log
  barrier} for increasing values of $t$, using the minimizer of one
iteration as the starting point for the next iteration. Let
$x^{\star}(t)$ be the minimizer of \eqref{eqn: log barrier}. A point
$x$ is strictly feasible if $h_i(x) < 0$ for all $i=1,\ldots,m$. The
barrier method is then given as: \\

\begin{mdframed}[nobreak=true,frametitle={\small Barrier method}]
\small
\setlength{\parskip}{0pt}
\texttt{Input: A strictly feasible $x$, $t = t^{(0)} > 0$, $\mu > 1$,
  $\epsilon > 0$ \\
\\
Repeat:
\begin{enumerate}[nolistsep,noitemsep]
\item
Centering step: Compute $x^{\star}(t)$ by minimizing $th_0 + \phi$
starting at $x$.
\item
Update: $x \mapsto x^{\star}(t)$.
\item
Stopping criterion: Quit if $m/t < \epsilon$.
\item
Increase $t$: $t \mapsto \mu t$.
\end{enumerate}
}
\end{mdframed}
\vskip 10pt

The barrier method solves \eqref{eqn: convex opt problem} with
accuracy no worse than $\epsilon$. The centering step is performed
using Newton's method with a backtracking line search. The following
theorem bounds the total number of Newton steps for the barrier method. 
\begin{theorem}[{\cite[p. 591]{boyd:convexOptimization2004}}] \label{thm:
    newton step bound}
Define the following constant:
\begin{align*}
C = \frac {10 - 4 \alpha}{\alpha \beta (1 - 2\alpha)^2} + \log_2
\log_2 (1/\epsilon). 
\end{align*}
Suppose that $th_0 + \phi$ is self-concordant and that the sublevel
sets of $h_0, \ldots, h_m$ are bounded. If $\mu = 1 + 1/\sqrt{m}$,
then the barrier method requires no more than 
\begin{equation*}
C \left( 1 + \log_2 \left( \frac{m}{t^{(0)} \epsilon} \right)
  \sqrt{m} \right)
\end{equation*}
Newton steps to solve \eqref{eqn: convex opt problem} to within
accuracy $\epsilon$. If $x^{\star}$ is the solution to \eqref{eqn:
  convex opt problem}, the barrier method returns a value $\tilde{x}^{\star}$
such that $|h_0(x^{\star}) - h_0(\tilde{x}^{\star})| \leq \epsilon$
and $h_i(\tilde{x}^{\star}) \leq 0$ for each $i = 1, \ldots, m$.
\end{theorem}

\begin{remark}
In particular, Theorem \ref{thm: newton step bound} applies to
quadratically constrained quadratic programs (QCQPs).
\end{remark}

\begin{remark}
The barrier method requires a starting point $x$ that is strictly
feasible. In general this requires solving for $x$, but we are only
interested in convex problems of the form 
\begin{align*}
\text{minimize} &\quad s \\
\text{subject to} &\quad h_i(x) \leq s, \quad i=1, \ldots, m.
\end{align*}
For this type of problem, finding a strictly feasible starting point
is simple. Indeed, $(x,s) \in \R^n \times \R$ is strictly feasible so
long as $s > \max_{i=1,\ldots,m} h_i(x)$. 
\end{remark}

The cost per iteration of Newton's method is the cost of computing the
Newton step plus the cost of the line search. Usually the cost of the
Newton step dominates. To compute the Newton step of a general
unconstrained convex optimization problem \eqref{eqn: unconstrained
  convex opt} requires solving the following system of equations:
\begin{equation} \label{eqn: newton equation}
H \Delta x = -g,
\end{equation}
where $H = \nabla^2h(x)$ and $g = \nabla h(x)$. Since $H$ is symmetric
and positive definite, one can use the Cholesky factorization of
$H$. This decomposes $H$ as $H = LL^T$, where $L$ is lower
triangular. One then solves $Lw = -g$ by forward substitution and $L^T
\Delta x = w$ by back substitution. If $H$ is a dense matrix with no
additional structure, then the total cost of computing the Newton step
is 
\begin{equation} \label{eqn: cost per Newton step}
D + \frac{1}{3}n^3 + 2n^2,
\end{equation}
where $D$ is the amount of work needed to compute $H$ and $g$,
$(1/3)n^3$ is the amount of work for the Cholesky factorization, and
$2n^2$ is the amount of work for the both the forward and back
substitution. 

If $H$ is sparse, sparse Cholesky factorization can be used in which
$H = P L L^T P^T$, where $P$ is a permutation matrix. The cost of the
factorization depends on the sparsity pattern, but can be much lower
(e.g., $O(n^{3/2})$). The cost is heavily dependent on the choice
of $P$; algorithms for finding good permutation matrices are known as
symbolic factorization methods. Sparsity can also be used to speed up
the forward and back substitutions. 

In the case of the barrier method for solving constrained
optimization problems, $h = th_0 + \phi$. In this case,
\begin{align}
H &= t \nabla^2 h_0(x) + \sum_{i=1}^m \frac{1}{h_i(x)^2} \nabla h_i(x)
\nabla h_i(x)^T - \sum_{i=1}^m \frac{1}{h_i(x)} \nabla^2
h_i(x), \label{eqn: H for barrier}\\
g &= t \nabla h_0(x) - \sum_{i=1}^m \frac{1}{h_i(x)} \nabla h_i(x). \nonumber
\end{align}
The worst case complexity of the Cholesky factorization and the forward and
back substitution remain the same, but we can further analyze the cost
of forming $H$ and $g$. Let $D'$ be the amount of work needed to
compute $\nabla h_i(x)$ and $\nabla^2 h_i(x)$ for all $i = 0, \ldots,
m$. Then
\begin{equation} \label{eqn: Newton step H cost for barrier}
D = D' + O(mn^2),
\end{equation}
where the $O(mn^2)$ term results from summing all of the terms in
\eqref{eqn: H for barrier}. 

\end{appendices}

\bibliography{/Users/matthewhirn/Dropbox/Mathematics/Bibliography/MainBib}

\projects{\noindent A.H.V. and F.M. were participants in the 2013
  Research Experience for Undergraduates (REU) at Cornell University
  under the supervision of M.H. During the REU program all three were
  supported by the National Science Foundation (NSF) grant number
  NSF-1156350. This paper is the result of work started during the
  REU. Additionally, all three attended the ``$8^{\text{th}}$ Whitney
  Problems Workshop,'' supported by and hosted at the Centre
  International de Rencontres Math\'{e}matiques (CIRM), during which
  time the paper was revised. During the writing of the manuscript
  M.H. was supported by European Research Council (ERC) grant
  InvariantClass 320959. He is currently partially supported by an Alfred
  P. Sloan Research Fellowship, a DARPA Young Faculty Award, and NSF
  grant \#1620216. F.M. is supported by a National Science Foundation
  Graduate Research Fellowship.}

\end{document}